\documentclass [12pt]{article}
\usepackage{amsmath,amssymb}
\usepackage{amssymb,verbatim}
\def\tcb{\textcolor{black}}

\def\rcb{\textcolor{black}}
\usepackage{enumerate}% http://ctan.org/pkg/enumerate

\newcommand{\interior}[1]{%
  {\kern0pt#1}^{\mathrm{o}}%
}
\usepackage[T1]{fontenc}
\usepackage[tableposition=top]{caption}
\usepackage[countmax]{subfloat}

\usepackage{caption} 
\usepackage{float}
\newcommand\preceqdot{\mathrel{\ooalign{$\preceq$\cr
  \hidewidth\raise0.125ex\hbox{$\cdot\mkern0.5mu$}\cr}}}
\newcommand\precdott{\mathrel{\ooalign{$\prec$\cr
  \hidewidth\raise0.015ex\hbox{$\cdot\mkern0.5mu$}\cr}}} 
\usepackage{mathabx}
\usepackage{float}
\usepackage{blindtext}
\usepackage{titlesec}
\title{Sections and Chapters}
\usepackage{enumitem}
\usepackage{multirow}
\usepackage{array}
\newcolumntype{L}[1]{>{\raggedright\let\newline\\\arraybackslash\hspace{0pt}}m{#1}}
\newcolumntype{C}[1]{>{\centering\let\newline\\\arraybackslash\hspace{0pt}}m{#1}}
\newcolumntype{R}[1]{>{\raggedleft\let\newline\\\arraybackslash\hspace{0pt}}m{#1}}
\usepackage{blkarray}
\usepackage{amsmath}
\usepackage{tikz-cd} 
\usepackage{subfig}
\usepackage{faktor}
\usepackage{xfrac}
\usepackage{faktor} 
\usepackage{amsmath} 
\usepackage{amssymb}
\usepackage{mathtools}

\usepackage[utf8]{inputenc}
\usepackage{graphicx}
\usepackage{graphicx,amssymb,amstext,amsmath}
\usepackage{tikz}
\usepackage{amsmath}
\usepackage{amssymb}
\usepackage{amsthm}
\usepackage{hebcal}
\usetikzlibrary{arrows, decorations.markings}
\usepackage{blindtext}
\usepackage{amsthm}
\usepackage[T1]{fontenc}
\usepackage[utf8]{inputenc}
\usepackage{blindtext}
\usepackage[T1]{fontenc}
\usepackage[utf8]{inputenc}
\newcommand\Item[1][]{%
  \ifx\relax#1\relax  \item \else \item[#1] \fi
  \abovedisplayskip=0pt\abovedisplayshortskip=0pt~\vspace*{-\baselineskip}}
\usepackage{tikz}
\usetikzlibrary{positioning, arrows.meta}

\usepackage{cite}
\newtheorem{thm}{Theorem}[section]
\newtheorem{lemma}[thm]{Lemma}
\newtheorem{prop}[thm]{Proposition}
\newtheorem{cor}[thm]{Corollary}
\newtheorem{conjecture}[thm]{Conjecture}

\newtheorem{lemmax}{Lemma}

\newtheorem{propx}{Proposition}

\newtheorem{example}[thm]{Example}
\newtheorem{remark}[thm]{Remark}
\newtheorem{question}[thm]{Question}

\newtheorem{question
}{Question}

\def\0{{\bf 0}}

\def\P{{\bf P}}

\def\R{{\bf R}}
\def\Z{{\bf Z}}

\usepackage{subfig}
\makeatletter
\def\keywords{\xdef\@thefnmark{}\@footnotetext}
  \title {\tcb{Exact expressions for the maximal probability that all $k$-wise independent bits are 1}}
 \date{}
\author{D. Berend\footnote{
Departments of Mathematics and Computer Science, Ben-Gurion
University, Beer Sheva 84105, Israel.
E-mail: berend@cs.bgu.ac.il }
\footnote{Research supported in part by
the Milken Families Foundation Chair in
Mathematics.}
\and
P. A. Ernst\footnote{Department of Mathematics, Imperial College London, United Kingdom. E-mail: p.ernst@imperial.ac.uk}
\and A. Kontorovich\footnote{Department of Computer Science, Ben-Gurion
University, Beer Sheva 84105, Israel. E-mail: karyeh@cs.bgu.ac.il}
\and
R. Kumar\footnote{Department of Mathematics, Ben-Gurion
University, Beer Sheva 84105, Israel. E-mail: kumarr@post.bgu.ac.il}
}
\begin{document}
\maketitle
\begin{abstract}
\tcb{Let $M(n, k, p)$ denote the maximum probability of the event $X_1 = X_2 = \cdots = X_n=1$ under a $k$-wise independent distribution whose marginals are Bernoulli random variables with mean $p$. A long-standing question is to calculate $M(n, k, p)$ for all values of $n,k,p$. This question has been partially addressed by several authors, primarily with the goal of answering asymptotic questions. The present paper focuses on obtaining exact expressions for this probability. To this end, we provide closed-form formulas of $M(n,k,p)$ for $p$ near 0 as well as $p$ near 1.}
%Suppose we have some distribution on $n$ bits, with the property that the bits are $k$-wise independent, for some $k$ between 2 and $n-1$. How far can the distribution be from the independent distribution on $n$ bits with the same marginals?

%Specifically, we deal here with the case where all bits are 1 with the same probability $p$, and ask how large can the probability that all bits are 1 be. 

\end{abstract}
\section{Introduction}
\tcb{Let $X_1, X_2,\ldots, X_n$ be a sequence of $n$ random bits, modelled as $n$ Bernoulli random variables with mean $p$ for some $p \in [0,1]$.} To specify the joint distribution of the random variables, one needs to specify the $2^n$ probabilities
$$\P(X_1= \varepsilon_1, X_2= \varepsilon_2,\ldots, X_n= \varepsilon_n), \qquad (\varepsilon_1,\varepsilon_2,\ldots, \varepsilon_n)\in \{0,1\}^n.$$
\tcb{The simplest setting for calculating the joint distribution occurs when the $n$ variables are independent}. Often, however, there is some form of weak independence between the $n$ bits. It then becomes natural to inquire what this weak dependence implies about the distribution of various functions of the $X_i$-s. This question was studied by Benjamini et al. \cite{benjamini-gurel-gurevich-peled} with respect to several natural functions. For simplicity, Benjamini et al. \cite{benjamini-gurel-gurevich-peled} assumed that all $X_i$-s have the same marginals, namely $\P(X_i=1)=p$, $1\leq i\leq n$, for some $p\in[0,1]$. One of the functions they studied is the product $\prod_{i=1}^n X_i$. \tcb{A vector $(X_1,\ldots,X_n)$ of random variables is $k$-wise independent if each subset of $k$ of the random variables is independent}. Denote by $A(n,k,p)$ the space of all probability measures on $\{0,1\}^n$ that make the $X_i$-s $k$-wise independent \tcb{with Bernoulli mean $p$ marginals}. In other words, $\P\in A(n,k,p)$ if 
$$\P(X_{i_1}= \cdots =X_{i_s}=1)= p^s,\qquad s\leq k,\, 1\leq i_1<\cdots <i_s\leq n.$$
\tcb{The authors of Benjamini et al. \cite{benjamini-gurel-gurevich-peled}} were especially interested in the maximum $M(n,k,p)$ of the functional $\P(X_1=X_2=\cdots= X_n=1)$ over the space $A(n,k,p)$. \tcb{Roughly speaking, they asked how large $k$ should be in order to ensure that this maximum is not much larger than the value of $p^n$ for the i.i.d. case.} This study was continued in \cite{peled-yadin-yehudayoff}, where the function $M(n,k,p)$ was specifically studied.

The main focus of \cite{benjamini-gurel-gurevich-peled} and \cite{peled-yadin-yehudayoff} \tcb{was to obtain both explicit bounds and asymptotic results} for $M(n,k,p)$.  For example, it was shown in \cite[Theorem 23]{benjamini-gurel-gurevich-peled} that for even $k$
$$M(n,k,p)\leq \frac{p^n}{\mathbb{P}(\textup{Bin}(n,1-p)\leq k/2)},$$
where $\mathbb{P}$ is the fully independent distribution on $n$ bits. Lower bounds on $M(n,k,p)$ were given in \cite[Theorem 1.1]{peled-yadin-yehudayoff}. \\
\indent \tcb{The main interest of the present paper is in obtaining exact results for $M(n,k,p)$.} Benjamini et al. \cite{benjamini-gurel-gurevich-peled} note that, for $k=2$ and $k=3$, closed-form expressions for $M(n,k,p)$ over the entire interval $[0,1]$ \tcb{can be obtained by invoking the closed-form results in \cite{Prekopa1}}  (see Propositions \ref{prop for M(n,2,p) except finitely many p} and \ref{prop for M(n,3,p) except finitely many p} below). An additional closed-form result, established by Benjamini et al.\cite[Proposition 35]{benjamini-gurel-gurevich-peled}, is that
\begin{equation}\label{results section tight bound for small p}
M(n,k,p)= p^k,\qquad 0\leq p\leq \frac{1}{n-1}.\end{equation} \tcb{We shall generalize this result in Theorem \ref{resuls sec, theorem on first half} of the present paper.}
\tcb{To the best of our knowledge, the result in \eqref{results section tight bound for small p} and the exact results for $M(n,2,p)$ and $M(n,3,p)$ constitute the only exact results in the literature for $M(n,k,p)$.}

\tcb{The primary purpose of the present paper is to help close this gap by obtaining further exact expressions for $M(n,k,p)$}. \tcb{This paper's key contributions are the closed-form formulas of $M(n,k,p)$ for $p$ near 0 (Theorems \ref{resuls sec, theorem on first half}, \ref{resuls sec, theorem on second interval 1}) and closed-form formulas for $p$ near 1 (Theorems \ref{resuls sec, theorem on last half}, \ref{resuls sec, theorem on second last interval 1}).} We also obtain tight upper bounds on $M(n,k,p)$ (see Theorem \ref{resuls sec, theorem on upper bounds}). \tcb{\rcb{The second purpose of this paper is to provide observations (based on numerical studies) about $M(n,k,p)$ which we conjecture will hold in general}. Our hope is that these observations will help future researchers with the very difficult task of providing exact formulas for $M(n,k,p)$ over the whole interval $[0,1]$.}

\tcb{A key motivation for the study of $k$-wise independent distributions arises from the extensive literature in computer science regarding derandomization of algorithms. The objective of derandomization is to analyze how a given algorithm behaves on $k$-wise independent data in comparison to how it would behave on completely independent data. The use of $k$-wise independent distributions for derandomization of algorithms was pioneered by \cite{Noga1986,Chor,Karp,Luby1985}. We refer the reader to \cite{Noga, Andrei, Kaplan, Karloff, Luby, Schmidt, Censor} for a more extensive overview of this literature.}

\tcb{The remainder of the paper is organized as follows.} In Section \ref{result section} we state our main results. In Section \ref{preliminaries k-wise} we recall \tcb{preliminaries needed for the main proofs}. In Section \ref{k-wise proof section} we provide these proofs. \tcb{\rcb{In Section \ref{numerical data}, \tcb{we conduct numerical studies and present our key observations as eight open conjectures.}}}

\section{The Main Results}\label{result section}
Let \tcb{$1\leq k\leq n$} be fixed integers and $p\in [0,1]$. Let $A(n,k,p)$ be the set of all laws~$\P$ making the collection $\{X_i\,: \, 1\leq i\leq n\}$ into $k$-wise  random bits with mean $p$. Let $M(n,k,p)$ be the maximum value of $\P(X_1=\cdots=X_n=1)$ over all laws $\P\in A(n,k,p)$. A law $\P\in A(n,k,p)$ is an {\it optimizing measure} for $(n,k,p)$ if it attains the maximum $M(n,k,p)$. 

Clearly, $M(n,k,0)=0$ and $M(n,1,p)=p$. Also, $M(n,k,1)=1$ and $M(n,n,p)=p^n$. Thus, we will implicitly assume that $2\leq k\leq n-1$ and $p\in (0,1)$. We first notice that the set $A(n,k,p)$ is convex since, if $\P_1$ and $\P_2$ both make $\{X_i\, :\, 1\leq i\leq n\}$ into $k$-wise random bits with mean $p$, then so does $\alpha \P_1 + (1-\alpha) \P_2$. As $A(n,k,p)$ is compact, there always exists an optimizing measure for $(n,k,p)$. The following lemma shows that we can restrict the search for an optimizing measure to a small subset of $A(n,k,p)$. For self-containedness, we also provide a proof. 

\begin{lemmax}\emph{\cite[p.507]{peled-yadin-yehudayoff}}\label{result sec, lemma for exchangeable measure}
There exists an optimizing measure for $(n,k,p)$ that makes $\{X_i\, :\, 1\leq i\leq n\}$ exchangeable.
\end{lemmax}
Recall that a sequence $C_1,C_2,\ldots,C_n$ of events is {\it exchangeable} if, for every $1\leq t\leq n$, the probabilities $\P(C_{i_1}\cap C_{i_2}\cap \cdots \cap C_{i_t})$ are the same for every $1\leq i_1<i_2<\cdots<i_t\leq n$.
\begin{proof}
Let $\P$ be an optimizing measure for $(n,k,p)$. Under the special case where all $k$-wise marginals have the same law, $k$-wise independence is closed under symmetrization. It follows that there is a measure $\P'$ that makes $\{X_i\,:\, 1\leq i\leq n\}$ exchangeable and gives $\sum_{i=1}^n X_i$ the same law as does $\P$. Thus $\P'(\sum_{i=1}^n X_i=n) =\P(\sum_{i=1}^n X_i=n)= M(n,k,p)$, so that $\P'$ is an optimizing measure for $(n,k,p)$.    
\end{proof}
By Lemma \ref{result sec, lemma for exchangeable measure}, the problem is the same if we restrict to exchangeable measures. Let $A_E(n,k,p)$ be the subset of $A(n,k,p)$, consisting all exchangeable measures. An exchangeable measure~$\P$ is uniquely 
determined by the quantities
\begin{equation}\label{result sec, equation of vi}
v_i= \P\left(\sum_{j=1}^i X_j=i,\sum_{j=i+1}^n X_j=0\right), \qquad 0\leq i\leq n.    
\end{equation}
By symmetry, for any set of $i$ bits, $v_i$ is the probability of the event that these bits are all~1 and the rest are all 0. There are $\binom{n}{i}$ such subsets, and hence $\sum_{i=0}^n \binom{n}{i}v_i=1$.

By contrast,
\begin{equation}\label{pre, sec, equation sum of vi equal to pi in the first form}
\P\left(\sum_{j=1}^i X_j=i\right)= \sum_{j=i}^n\binom{n-i}{j-i}v_j=p^i,\qquad 0\leq i\leq k,
\end{equation}
but this relation is no longer (necessarily) true for $i>k$.
Thus, the set of all $k$-wise independent distributions, whose marginals are all defined by $\P(X_i=1)=p$, obeys the $k+1$ equality constraints \tcb{in} \eqref{pre, sec, equation sum of vi equal to pi in the first form}, in addition to the non-negativity constraints. 
% Put
% $$w_i= \P\left(\sum_{j=1}^n X_j =i\right), \qquad i=0,1,2,\ldots,n,$$
% and note that
% $$w_i= \binom{n}{i} v_i,\qquad i=0,1,2,\ldots, n.$$
A~subset $ I\subseteq \{0,1,\ldots,n\}$ is a {\it weak support} of an exchangeable measure $\P$ if $v_i=0$ for all $i\in \{0,1,\ldots,n\}- I$. The set $I$ is the {\it support} of $\P$ if $v_i>0$ for each $i\in I$ and  $v_i=0$ for all $i\in \{0,1,\ldots,n\}- I$.  
The following two propositions provide explicit formulas for $M(n,2,p)$ and $M(n,3,p)$. These formulas were obtained in \cite{Prekopa1} (see also \cite{benjamini-gurel-gurevich-peled} and \cite{Arjun}). 
For completeness, we also present the corresponding optimizing measures as in \cite{Arjun}.

\begin{propx}\emph{\cite[Theorem 4.1]{Arjun}}\label{prop for M(n,2,p) except finitely many p}
Let $n\geq 3$ and $p\in \left[\frac{j-1}{n-1},\frac{j}{n-1}\right]$ for some $1\leq j\leq n-1$. Then,
\begin{equation}\label{result sec, equation of M(n,2,p) in the lemma}
M(n,2,p)= \frac{1}{\binom{n-j+1}{2}}\left(\binom{n}{2}p^2 -n(j-1)p + \binom{j}{2}\right).
\end{equation}
Moreover, the optimizing distribution is supported on $\{j-1,j,n\}$ and given by
{ \everymath={\displaystyle}
\begin{equation}\label{result sec, equation distribution of v(j) in lemma}
 \begin{split}
 v_i=\begin{cases}
     \frac{1}{\binom{n-1}{j-1}}\cdot (1-p)\cdot (j+p-np),&\qquad i=j-1,\\
     \\
     \frac{1}{\binom{n-1}{j}} \cdot (1-p)\cdot (np-p-j+1),&\qquad i=j,\\
     \\
  \frac{1}{\binom{n-j+1}{2}}\cdot\left(\binom{n}{2}p^2 -n(j-1)p + \binom{j}{2}\right),&\qquad i=n.
 \end{cases} 
 \end{split}   
\end{equation}
}
\end{propx}

\tcb{For self-containedness, we provide a proof of the proposition in Section \ref{k-wise proof section}}. Once we know the value of $M(n,k,p)$ for a certain even $k$ (for all $n$), \tcb{we will also know its value when $k$ is replaced by $k+1$}. In fact (see \cite[p.504]{peled-yadin-yehudayoff})
\begin{equation}\label{intro, equation M(n,k,p) for odd k}
M(n,k+1,p)= p M(n-1,k,p),  \qquad  k\equiv 0 \, (\textup{mod}\, 2),\, n\geq k-2,\, p\in [0,1].  
\end{equation}
In particular, Proposition \ref{prop for M(n,2,p) except finitely many p} readily yields a formula for $M(n,3,p)$. In Proposition \ref{prop for M(n,3,p) except finitely many p} we also provide the optimizing measure.
\begin{propx}\label{prop for M(n,3,p) except finitely many p}
Let $n\geq 4$ and $p\in \left[\frac{j-1}{n-2},\frac{j}{n-2}\right]$ for some $1\leq j\leq n-1$. Then, 
\begin{equation}\label{result sec, equation of M(n,3,p) in the lemma}
M(n,3,p)=  p\cdot \frac{1}{\binom{n-j}{2}}\left(\binom{n-1}{2}\cdot p^2 -(n-1)(j-1)p + \binom{j}{2}\right).
\end{equation}
Moreover, the optimizing distribution is supported on $\{0, j,j+1,n\}$ and given by
{ \everymath={\displaystyle}
\begin{equation}\label{result sec, equation distribution of v(j) in M(n,3,p) lemma}
\begin{split}
 w_i=\begin{cases}
\frac{1}{\binom{j+1}{2}}\cdot (1-p)\left(\binom{n-1}{2}p^2-j(n-1)p + \binom{j+1}{2})\right),&\qquad i=0,\\
\\
 p\cdot \frac{1}{\binom{n-2}{j-1}}\cdot  (1-p)(j-(n-2)p),&\qquad i=j,\\
  \\
     p\cdot \frac{1}{\binom{n-2}{j-1}}\cdot (1-p)(1-j + (n-2)p),&\qquad i=j+1,\\
     \\
     p\cdot \frac{1}{\binom{n-j}{2}}\left(\binom{n-1}{2}\cdot p^2 -(n-1)(j-1)p + \binom{j}{2}\right),&\qquad i=n.
 \end{cases} 
 \end{split}
\end{equation}
}
\end{propx}
\tcb{Since the proof of the proposition is similar to that of Proposition \ref{prop for M(n,2,p) except finitely many p}, it is omitted.}

It will follow from the discussion in Subsection \ref{subsection on LP} that $M(n,k,p)$ is a piecewise polynomial function of degree $k$ in $p$ for every fixed $n$ and $k$. In Proposition \ref{prop for M(n,2,p) except finitely many p} and~\ref{prop for M(n,3,p) except finitely many p}, these functions \tcb{are} presented explicitly for $k=2,3$. We will refer to the (maximal) intervals on which $M(n,k,p)$ is polynomial, starting at $p=0$ and continuing towards $p=1$, as the {\it first interval} of $M(n,k,p)$, {\it second interval}, and so forth. Notably, for $k=2$ and $k=3$, the intervals have constant lengths of $\frac{1}{n-1}$ and $\frac{1}{n-2}$, respectively. We will see that the intervals are of varying lengths for $k\geq 4$.

For integers $m\geq l\geq 0$, denote
\begin{equation}\label{proof sec, defi of bonferroni}
B(m,l,x)= 1-\binom{m}{1}x + \binom{m}{2}x^2 - \cdots + (-1)^l \binom{m}{l}x^l, \qquad  x\in \R
\end{equation}
\tcb{where} $B(m,l,x)=0$ for negative integers $l$.
We refer to these polynomials as {\it Bonferroni polynomials} for reasons that will be explained in Subsection \ref{subsection Bonferroni}.

The following theorem provides upper bounds on $M(n,k,p)$ in terms of Bonferroni polynomials.

\begin{thm}\label{resuls sec, theorem on upper bounds}\hspace{0.025 cm}
\begin{description}
\item{1.} For every $k$:
$$M(n,k,p)\leq \frac{1}{\binom{n-k+m-1} {m}} 
 \cdot 
 p^{k-m}\cdot B(n-k+m,m,p), \qquad 0\leq m\leq k,\, m\equiv 0 \, (\textup{mod}\, 2).
$$
\item {2.} For even $k$:
$$M(n,k,p)\leq  B(n,k,1-p).$$    
\end{description}
\end{thm}

As mentioned in \eqref{results section tight bound for small p}, it was proven in Benjamini et al. \cite[Proposition 35]{benjamini-gurel-gurevich-peled} that, if $p\leq \frac{1}{n-1}$, then $M(n,k,p)= p^k$. \tcb{We generalize this result in Theorem \ref{resuls sec, theorem on first half} below.}

\begin{thm}\label{resuls sec, theorem on first half}
$M(n,k,p)=p^k$ if and only if $p\leq \frac{1}{n-k+1}$. Moreover, the optimizing distribution for such $p$ is given by
\begin{align}\label{result ses, equation of optimal distribution in the first interval}
v_i=\begin{cases}
 (1-p)\cdot p^i \cdot B(n-i-1, k-i-1,p),\quad &0\leq i\leq k-1,\\ 
0, \quad &k\leq i\leq n-1,\\
p^k, \qquad &i=n.
\end{cases}
\end{align}
\end{thm}

For general $k$, Theorem \ref{resuls sec, theorem on first half} \tcb{states the value of} $M(n,k,p)$ on the first interval (i.e., near~0). \tcb{We will also find the value of $M(n,k,p)$ on the second interval (Theorem \ref{resuls sec, theorem on second interval 1}) and on the last two intervals (Theorem \ref{resuls sec, theorem on last half} and \ref{resuls sec, theorem on second last interval 1}).}

\begin{thm}\label{resuls sec, theorem on last half}
Let $2\leq k\leq n-1$ be positive integers with $k$ even. Then
\begin{equation}
M(n,k,p)= B(n,k,1-p),\qquad 1-\frac{1}{n-k+1}\leq p\leq1.
\end{equation}
Moreover, the optimizing distribution is given by
\begin{align}\label{result sec, equation of optimal distribution in the last interval}
 v_i= \begin{cases}
     0,\quad &0\leq i\leq n-k-1,\\
     (1-p)^{n-i}\cdot B(i ,i-(n-k),1-p),\quad &n-k\leq i\leq n.\\
 \end{cases}    
\end{align}
\end{thm}

From Theorem \ref{resuls sec, theorem on first half} and \ref{resuls sec, theorem on last half}, we obtain
\begin{cor}
\begin{equation*}
\begin{split}
 M(n,n-1,p)= \begin{cases}
  p^{n-1},&\qquad p\leq 1/2,\\
  B(n,n-1,1-p),& \qquad p> 1/2.\\
 \end{cases}   
\end{split}    
\end{equation*}
For $p\leq1/2$, the optimal distribution is given by \eqref{result ses, equation of optimal distribution in the first interval}. For $p>1/2$, it is given by \eqref{result sec, equation of optimal distribution in the last interval}.
\end{cor}
\tcb{We now proceed to state Theorem \ref{resuls sec, theorem on second interval 1}.}
\begin{thm}\label{resuls sec, theorem on second interval 1}
\hspace{0.025 cm}
\begin{description}
\item {1.} Let $2\leq k\leq n-2$, and 
$$p_1= \frac{1}{n-k+1}.$$
Then, for some $p_2 >p_1$,
\begin{equation}\label{result sec, equation of M(n,k,p) in second interval}
M(n,k, p)= \frac{1}{\binom{n-k+1}{2}}\cdot p^{k-2}\cdot  B(n-k+2,2,p),\qquad p\in \left[p_1,p_2\right]. 
\end{equation}
Moreover, the optimizing distribution is given by
\begin{equation}\label{result sec, equation2 of M(n,k,p) in second interval}
 v_i= \begin{cases}
     p^i B(n-i,k-i,p) + (-1)^{k-i-1} \frac{\binom{n-i-1}{k-i}}{\binom{n-k+1}{2}}p^{k-2} B(n-k+2,2,p),\quad &i\in I-\{n\},\\
     0, \quad &i \not\in I,\\
     \frac{1}{\binom{n-k+1}{2}}\cdot p^{k-2}\cdot  B(n-k+2,2,p), \quad &i=n,\\
 \end{cases}
\end{equation}
where $I= \{0,1,\ldots ,k-3, k-1,k,n\}$.
\item {2.} The largest $p_2$ for which \eqref{result sec, equation of M(n,k,p) in second interval} and \eqref{result sec, equation2 of M(n,k,p) in second interval} hold is
\begin{itemize}
\item $\frac{2}{n-1}$ $-$ for $k=2$;
\item $\frac{2}{n-2}$ $-$ for $k=3$;
\item the zero of the cubic polynomial
\begin{equation}\label{result sec, equation of f(2,k-4)}
1-\binom{n-k+3}{1}p + \frac{5}{6} \binom{n-k+3}{2}p^2- \frac{1}{2}\binom{n-k+3}{3}p^3   
\end{equation}
in the interval $\left[\frac{1}{n-k+1},\frac{2}{n-k+1}\right]$ $-$ for $k\geq 4$.
\end{itemize}
\end{description}
\end{thm}

We mention that, using Lemma \ref{prelim, lemma for Bonferroni polynomial}$.1$ below, we may rewrite the first case in \eqref{result sec, equation2 of M(n,k,p) in second interval} alternatively in the form
$$v_i=
   (1-p) \left(p^i B(n-i-1,k-i-1,p)
   + (-1)^{k-i-1} \frac{\binom{n-i-1}{k-i}}{\binom{n-k+1}{2}}p^{k-2} B(n-k+1,1,p)\right),$$
for $i\in I-\{n\}$.

\begin{remark}\emph{
In view of Theorems \ref{resuls sec, theorem on first half} and \ref{resuls sec, theorem on second interval 1}, the upper bound in Theorem \ref{resuls sec, theorem on upper bounds}$.1$ is tight for $m=0,2$, and by Theorem \ref{resuls sec, theorem on last half} the same holds for Theorem \ref{resuls sec, theorem on upper bounds}$.2$. $M(n,k,p)$ is given by the bound with $m=0$ in Theorem \ref{resuls sec, theorem on upper bounds}$.1$ on the first interval, by the bound with $m=2$ on the second interval, and by the bound in Theorem \ref{resuls sec, theorem on upper bounds}$.2$ on the last interval. In all of the examples we have checked, $M(n,k,p)$ coincided with the bound for $m=4$ on the third interval. \tcb{However,   $M(n,k,p)$ only sometimes coincided with the bound for $m=6$ on the fourth interval}. For example, for $n=10$, $k=6$, in the first four intervals, $M(n,k,p)$ is given by the bounds as in Theorem \ref{resuls sec, theorem on upper bounds}$.1$ for $m=0,2,4,6$, respectively,  but for $n=11$, only the first three are as in Theorem \ref{resuls sec, theorem on upper bounds}$.1$, but not the 4-th, as follows:
\begin{equation}\label{result section, equation for M(11,6,p)}
\begin{split}
 M(11,6,p)= \begin{dcases}
p^6,& 0 < p < \frac{1}{6},\\
 \frac{p^4}{15} (1 - 7 p + 21 p^2),&\frac{1}{6} < p < 0.263\ldots,\\
 \frac{p^2}{70} (1 - 9 p + 36 p^2 - 84 p^3 + 126 p^4),&0.263\ldots < p < 0.294\ldots,\\
 \frac{p^2}{140}(5 - 45 p + 180 p^2 - 378 p^3 + 378 p^4),&0.294\ldots < p < 0.315\ldots.
 \end{dcases}   
\end{split}    
\end{equation}
The first three equalities in \eqref{result section, equation for M(11,6,p)} are the same as the bound in Theorem \ref{resuls sec, theorem on upper bounds}$.1$, namely 
$$\frac{1}{\binom{n-k+m-1}{m}}\cdot p^{k-m}B(n-k+m,m,p),$$
for $m=0,2,4$, respectively,
while the fourth equality is different from that bound for $m=6$, which is
$$\frac{1}{\binom{n-k+m-1}{m}}\cdot p^{k-m}B(n-k+m,m,p)= \frac{1}{120} (1 - 11 p + 55 p^2 - 165 p^3 + 330 p^4 - 462 p^5 + 462 p^6).$$
}
\end{remark}

The following theorem deals with the \tcb{second to last interval}.
\begin{thm}\label{resuls sec, theorem on second last interval 1}\hspace{0.025 cm}
\begin{description}
\item {1.} For $k$ even let $2\leq k\leq n-2$ be positive integers and 
$$\overline{p}_1= 1-\frac{1}{n-k+1}.$$
Then, for some $\overline{p}_2 <\overline{p}_1$,
\begin{equation}\label{result sec, equation of M(n,k,p) in second last interval}
M(n,k,p)= B(n,k,1-p)+ \frac{\binom{n}{k+1}}{\binom{n-k+1}{2}}\cdot (1-p)^{k-1}B(n-k+1,1,1-p), \qquad p\in \left[\overline{p}_2,\overline{p}_1\right].
\end{equation}
Moreover, the optimizing distribution is given by
\begin{equation}\label{result sec, equation of distribution v(i) in second last interval}
\begin{split}
 v_{n-i}= \begin{cases}
     (1-p)^i B(n-i,k-i,1-p)-\\
     ~~(-1)^{k-i-1} \frac{\binom{n-i}{k+1-i}}{\binom{n-k+1}{2}}\cdot(1-p)^{k-1} B(n-k+1,1,1-p),& \quad i\in (n-I),\\
     0 & \quad i \not\in n- I,\\
 \end{cases}
\end{split}
\end{equation}
where $I= \{n-k-1,n-k,n-k+2,\ldots ,n-1,n\}$.
\item {2.} The smallest $\overline{p}_2$ for which \eqref{result sec, equation of M(n,k,p) in second interval} and \eqref{result sec, equation of distribution v(i) in second last interval} hold is
\begin{itemize}
\item $\frac{2}{n-1}$ $-$ for $k=2$;
\item the zero of the cubic polynomial
\begin{equation}\label{result sec, equation2 of f(2,k-4)}
1-\binom{n-k+3}{1}(1-p) + \frac{5}{6} \binom{n-k+3}{2}(1-p)^2- \frac{1}{2}\binom{n-k+3}{3}(1-p)^3    
\end{equation}
in the interval $\left[1-\frac{2}{n-k+1},1-\frac{1}{n-k+1}\right]$ $-$ for $k\geq 4$.
\end{itemize}
\end{description}
\end{thm}
\tcb{We conclude this section with two remarks.}
\begin{remark}\emph{
\begin{enumerate}
\item Unlike the polynomials describing $M(n,k,p)$ on the first, second, and last intervals, \tcb{which form upper bounds on $M(n,k,p)$ on the entire interval $[0,1]$}, the polynomial on the right-hand side of \eqref{result sec, equation of M(n,k,p) in second last interval} is not an upper bound \tcb{for} $M(n,k,p)$. In all of the examples we have checked, the polynomial is negative on some sub-interval of $[0,1]$.
\item It follows from \eqref{result sec, equation of f(2,k-4)} and \eqref{result sec, equation2 of f(2,k-4)} that the smallest $\overline{p}_2$ in Theorem \ref{resuls sec, theorem on second last interval 1} and the largest $p_2$ in Theorem \ref{resuls sec, theorem on second interval 1} are symmetric with respect to the point $p=1/2$, namely $\overline{p}_2= 1- p_2$.
\end{enumerate}
}
\end{remark}

\begin{remark}\emph{
Theorems \ref{resuls sec, theorem on last half}, \ref{resuls sec, theorem on second last interval 1}, and the numerical data suggest that, for even $k$ between $2$ and $n-2$, there exists a point $\overline{p}_3<\overline{p}_2$
such that
$$M(n,k,p)= B(n,k,1-p) + \frac{\binom{n}{k+1}}{\binom{n-k+3}{4}}(1-p)^{k-3}B(n-k+3,3,1-p),\qquad p\in [\overline{p}_{3}, \overline{p}_{2}],$$
where the endpoints $\overline{p}_2, \overline{p}_3$ are symmetric, respectively, to the endpoints $p_2, p_3$ of the third interval with respect to the point $p=1/2$.
}
\end{remark}
\tcb{This concludes the section of the main results, all of which will be proven in  Section~\ref{k-wise proof section}.}
\section{Preliminaries}\label{preliminaries k-wise}
\tcb{In this section, we recall the necessary preliminaries for the proofs in Section \ref{k-wise proof section}.}
\subsection{Bonferroni Inequalities}\label{subsection Bonferroni}
Let $A_1, A_2,\ldots, A_n$ be arbitrary events in a probability space $(\Omega, \mathcal{B},\P)$. Let $N$ be the (random) number of events $A_i$-s that occur. 
Denote
\begin{equation}\label{pre, equation of Si for arbitrary events}
S_t= \sum_{1\leq i_1<i_2<\cdots<i_t\leq n} \P(A_{i_1}\cap A_{i_2}\cap\cdots\cap A_{i_t}),\qquad t=1,2,\ldots, n,  
\end{equation}
\tcb{where $S_0=1$}.
{\it Bonferroni's inequalities}, first proved by Bonferroni \cite{Bonferroni} (see also \cite{Spencer}), state that
\begin{equation}\label{pre, equation of  Bonferroni inequalities}
\sum_{i=0}^{2m+1}(-1)^i S_i\leq \P(N=0)\leq   \sum_{i=0}^{2m}(-1)^i S_i,\qquad m\geq 0.  
\end{equation}
Inequalities of the form
\begin{equation}\label{pre, equation of  Bonferroni type inequalities}
\sum_{i=0}^{n}c_i S_i\leq \P(N=r)\leq   \sum_{i=0}^{n}d_i S_i, 
\end{equation}
where $c_i= c_i(r,n)$ and $d_i= d_i(r,n)$ are constants (unrelated to the specific events $A_j$, $1\leq j\leq n$), possibly zero, are {\it Bonferroni-type inequalities} (see \cite{Galambos1}). Such inequalities have a long history (cf. \cite{Bonferroni, Prekopa2, Galambos1, M Lee, Galambos2}). An extension of \eqref{pre, equation of  Bonferroni inequalities} is given by {\it Jordan's inequalities}, which are as follows (cf. \cite[p.8]{M Lee})
\begin{equation*}\label{pre, equation of  Jordan inequalities}
\sum_{i=0}^{2m+1}(-1)^i \binom{i+r}{r}S_{i+r}\leq \P(N=r)\leq   \sum_{i=0}^{2m}(-1)^i\binom{i+r}{r} S_{i+r},\qquad 0\leq r\leq n,\, m\geq 0.  
\end{equation*}

It follows from \cite[Theorem 2]{Galambos1} that \eqref{pre, equation of  Bonferroni type inequalities} holds for arbitrary events $A_1,A_2,\ldots,A_n$ if and only it holds for exchangeable events. In other words, \tcb{the inequality in} \eqref{pre, equation of  Bonferroni type inequalities} holds for arbitrary $A_1,A_2,\ldots, A_n$ if and only if
\begin{equation*}
\sum_{i=0}^n c_i\binom{n}{k}P_k\leq \P(N'=r)\leq  \sum_{i=0}^n d_i\binom{n}{k}P_k  \end{equation*}
for arbitrary exchangeable events $C_1,C_2,\ldots,C_n$, where $N'$ is the number of events $C_i$ that occur and $P_k= \P(C_{i_i}\cap C_{i_2}\cap\cdots \cap C_{i_k})$. Furthermore, the \tcb{inequality in} \eqref{pre, equation of  Bonferroni type inequalities} holds for arbitrary events $A_1, A_2,\ldots, A_n$ if and only if it holds when the $A_j$'s are independent and $P(A_j)=p$ for each $1\leq j\leq n$ (see \cite[Theorem 3]{Galambos1}).
\vspace{0.25 cm}

Let $X_1,X_2,\ldots,X_n$ be $\P$-distributed random bits, where $\P\in A_E(n,k,p)$, namely $\P(X_i=1)=p$ and the $X_i$-s are $k$-wise independent. Let $A_i=\{X_i=1\}$, $1\leq i\leq n$. By~\eqref{pre, equation of Si for arbitrary events},
\begin{equation}\label{pre sec, equation of Si in term of (1-p)}
 S_t= \binom{n}{t}p^t,\qquad t=1,2,3,\ldots,k.
 \end{equation}
\tcb{Employing \eqref{pre, equation of  Bonferroni inequalities} and \eqref{pre sec, equation of Si in term of (1-p)}, we obtain}
\begin{equation}\label{pre sec, equation of upper bound of P(x1,..Xn) as expresion}
\begin{split}
\P(N=0)&\begin{cases}
 \leq B(n,k,p)=1-\binom{n}{1}p + \cdots + \binom{n}{k}p^k, & \quad k \equiv 0\, (\textup{mod}\, 2),\\
 \geq B(n,k,p)= 1-\binom{n}{1}p + \cdots - \binom{n}{k}p^k,&\quad k \equiv 1\, (\textup{mod}\, 2),\\
\end{cases} 
\end{split}
\end{equation}
which explains the term ``Bonferroni polynomials'' introduced in Section \ref{result section}. \tcb{We refer the reader to} \cite{Galambos2} for further details on Bonferroni inequalities.

Several useful properties of Bonferroni polynomials are listed in the following lemma. \tcb{So as to ease the flow of presentation, the proof is postponed to Section \ref{k-wise proof section}.}
\begin{lemma}\label{prelim, lemma for Bonferroni polynomial}
For any positive integers $0\leq k\leq n$,
\begin{description}
\item {1.} \vspace{-0.8 cm} \begin{equation*}
\begin{split}
 B(n,k,p)&= B(n-1,k,p)- pB(n-1,k-1,p)\\  
 &=(1-p)B(n-1,k-1,p)+ (-1)^k \binom{n-1}{k}p^k.
\end{split}    
\end{equation*}
\item {2.} \vspace{-0.6 cm} $$\frac{d}{d p} B(n,k,p)= -nB(n-1,k-1,p).$$
\item {3.}
\vspace{-0.6 cm}\begin{equation}\label{prilim sec, equation for Bonferroni polynomial in terms of others}
B(n,k,p)= B(k,k,p)- \sum_{i=0}^{k-1}\binom{n-k}{k-i} p^{k-i} B(n-k+i,i,p).    
\end{equation}
\end{description}
\end{lemma}

\tcb{Another} essential inequality concerning the Bonferroni polynomials, crucial for us in the sequel, is presented in the following proposition. \tcb{So as to ease the flow of presentation, its proof is postponed to Section \ref{k-wise proof section}.}

\begin{prop}\label{result, proposition for B(n,l,p)}
Let $0\leq l\leq m$. Then
\begin{equation}\label{result sec, positivity of B(n,l,p) in first half}
B(m,l,p)\geq 0, \qquad  0\leq p\leq \frac{1}{m-l+1}.    
\end{equation}
Moreover, the inequality in \eqref{result sec, positivity of B(n,l,p) in first half} is strict except for the following cases:
\begin{description}
\item {1.} $l=1$ and $p=1/m$.
\item {2.} $m$ is odd, $l=m-1$, and $p=1/2$.
\item {3.} $l=m$ and $p=1$.
\end{description}
\end{prop}

\vspace{0.25 cm}

\subsection{$M(n,k,p)$ as the Solution of a Linear Programming Problem}\label{subsection on LP}
Let $X_1,X_2,\ldots,X_n$ be $\P$-distributed random bits, where $\P\in A_E(n,k,p)$. Let $A_i=\{X_i=1\}$, $1\leq i\leq n$. Let $N$ be the number of events $A_i$ \tcb{which} occur.
By \eqref{result sec, equation of vi} and \eqref{pre, sec, equation sum of vi equal to pi in the first form}, \tcb{one may recast the problem of calculating $M(n,k,p)$} for given $n,k,p$ as a primal linear programming problem, as follows:
\begin{equation}\label{pre sec, equation of LP primal in k-wise independent}
\begin{array}{ccccc}
 \textup{maximize} &v_n && &\\
\vspace{0.25 cm}
\textup{subject to}& &&&\\
 & \displaystyle \sum_{j=0}^n \binom{n-i}{j-i} v_j&=&p^i,& \qquad i=0,1,2,\ldots, k,\\
\vspace{0.25 cm}
 & v_j&\geq &0,& \qquad j=0,1,2,\ldots,n.
\end{array}    
\end{equation}
\tcb{Letting}
$$w_i= \binom{n}{i}v_i, \qquad i=0,1,2,\ldots,n,$$
and
$$S_i= \binom{n}{i}p^i, \qquad i=0,1,2,\ldots,k,$$ 
we may rewrite \eqref{pre sec, equation of LP primal in k-wise independent} in the form:
\begin{equation}\label{pre sec, equation of LP primal}
\begin{array}{ccccc}
 \textup{maximize} &w_n && &\\
\vspace{0.25 cm}
 \textup{subject to}& &&&\\
 & \displaystyle \sum_{j=0}^n \binom{j}{i} w_j&=&S_i,& \qquad i=0,1,2,\ldots, k,\\
\vspace{0.25 cm}
 & w_j&\geq &0,& \qquad j=0,1,2,\ldots,n.
\end{array}    
\end{equation}
\indent \tcb{We now pause to introduce some notation which will enable us to rewrite} \eqref{pre sec, equation of LP primal} in a more compact form.
\begin{equation}\label{pre sec, equation for notation of ai and b and A}
\begin{array}{cclc}
a_i&=& \displaystyle \left(1,i,\binom{i}{2},\ldots, \binom{i}{k}\right)^T \in \R^{k+1},&\qquad i=0,1,2,\ldots,n, \\
\vspace{0.25 cm}
A&=& \displaystyle [a_0,a_1,\ldots, a_n]\in M_{(k+1)\times (n+1)},&  \\
\vspace{0.25 cm}
w&=& \displaystyle (w_0,w_1,\ldots,w_n)^T,&\\
\vspace{0.25 cm}
c&=& \displaystyle (0,0,\ldots,0,1)^T\in \R^{n+1},\\
\vspace{0.25 cm}
d&=& \displaystyle (S_0,S_1,\ldots,S_k)^T\in \R^{n+1},
\end{array} 
\end{equation}
where $T$ denotes matrix transpose.
With \tcb{this notation}, \eqref{pre sec, equation of LP primal} becomes
\begin{equation}\label{pre sec, equation of LP primal in matrix form}
\begin{array}{ccccc}
 \textup{maximize}&c^{T}\cdot w && &\\
\vspace{0.25 cm}
 \textup{subject to}& &&&\\
 & A\cdot w&=&d,&\\
\vspace{0.25 cm}
 & w&\geq &0.&
\end{array}    
\end{equation}
The dual of \eqref{pre sec, equation of LP primal in matrix form} is
\begin{equation}\label{pre sec, equation of LP dual in matrix form}
\begin{array}{ccccc}
 \textup{minimize}&d^{T}\cdot y && &\\
\vspace{0.25 cm}
 \textup{subject to}& &&&\\
 & A^T\cdot y&\geq&c.&\\
\end{array}    
\end{equation}

\tcb{We now mention some definitions and recall some key results} of Boros and Pr{\'e}kopa \cite{Prekopa1}, and Pr{\'e}kopa \cite{Prekopa2,Prekopa3}, who considered a more general problem than \eqref{pre sec, equation of LP primal in matrix form}.
\tcb{With slight abuse of notation}, let $B$ be a basis consisting of $k+1$ columns of $A$. \tcb{Then} $B$ is {\it dual feasible} if 
\begin{equation}\label{pre sec, equation of condition for dual nondegenerate}
c_B^T B^{-1}a_j\geq c_j, \qquad j\in \{0,1,\ldots,n\}-I,   \end{equation}
where $I$ is the set of subscripts of the basis vectors and $c_B$ is the vector of basic components of $c$. The basis $B$ is {\it dual nondegenerate} if the inequality in \eqref{pre sec, equation of condition for dual nondegenerate} is strict for every $j$.

Let $I= \{i_0,i_1,\ldots,i_k\}$, $0\leq i_0<i_1<\ldots<i_k\leq n$. Denote
$$B(I)= [a_{i_0},a_{i_1},\ldots, a_{i_k}].$$
It follows from \cite[Theorem 7.1]{Prekopa1} that 
$$B(I)^{-1}= [b_{s,t}]_{s\in I}^{t=0,1,\ldots,k},$$ where
\begin{equation}\label{pre sec, equation of elements of matrix B}
b_{s,t} =  \frac{(-1)^t \sum_{\alpha=0}^t (-1)^\alpha \binom{t}{\alpha}\prod_{u\in I-\{s\}}(u- \alpha)}{\prod_{u\in I-\{s\}}(u- s)},\qquad s\in I,\, 0\leq t\leq k. \end{equation}

\begin{thm}\emph{\cite[Theorem 7.2]{Prekopa1}}\label{pre sec, theorem on I} A basis $B(I)$, corresponding to problem \eqref{pre sec, equation of LP primal in matrix form}, is dual feasible if and only if $I$ satisfies the following conditions:
\begin{enumerate}
\item $n\in I$.
\item If $k$ is odd then $0\in I$.
\item The elements of $I-\{n\}$ in the case of even $k$, and of $I-\{0,n\}$ in the case of odd~$k$, come in pairs of consecutive numbers. Namely,
\begin{equation}\label{pre sec, equation of I in theorem}
\begin{split}
I= \begin{cases}
\{i_1,i_1+1,i_2,i_2+1,\ldots, i_{k/2},i_{k/2}+1,n\},&\qquad k \equiv 0\, (\textup{mod}\, 2),\\ 
\{0, i_1,i_1+1,i_2,i_2+1,\ldots, i_{(k-1)/2},i_{(k-1)/2}+1,n\},&\qquad k \equiv 1\, (\textup{mod}\, 2),\\
\end{cases}    
\end{split}    
\end{equation}
where $i_{j+1}\geq i_{j}+2$ for each $j$.
\end{enumerate}
\end{thm}

\begin{thm}\emph{\cite[Theorem 7.3]{Prekopa1},\cite[Theorem IV.1.]{Galambos2}}\label{pre sec, theorem on B(I) tight bound}
Let $I\subseteq \{0,1,2,\ldots,n\}$ be a set of size $k+1$ of the type in \eqref{pre sec, equation of I in theorem}. Then 
\begin{equation}\label{proof sec, equation of P(u=n)}
M(n,k,p)\leq \sum_{i=0}^k b_{n,i} S_i,
\end{equation}
where the $b_{n,i}$ are as in \eqref{pre sec, equation of elements of matrix B}. Moreover, If 
$B(I)^{-1}d\geq 0$, the bound in \eqref{proof sec, equation of P(u=n)} is tight, and therefore
$$M(n,k,p)= \sum_{i=0}^k b_{n,i} S_i.$$
\end{thm}

\begin{thm} \emph{\cite[Theorem 6]{Prekopa3}}\label{pre sec, theorem on basis change}
Let $B(I)$ be a dual feasible basis corresponding to \eqref{pre sec, equation of LP primal in matrix form}. For $i\in I$, if all the bases $B(I_{i,j})$, where
$$I_{i,j}= (I-\{i\})\cup \{j\}, \qquad j\in \{0,1,\ldots,n\}-I,$$
are not dual feasible basis, then we will have $w_i\geq 0$.
\end{thm}

All dual feasible bases for \eqref{pre sec, equation of LP primal in matrix form} are dual nondegenerate \cite[p.232]{Prekopa3}. Moreover, the optimal distribution is unique \cite[Theorem 5]{Prekopa3}. We refer the reader to \cite[Chapter IV]{Galambos2} for more details on dual feasibility of the basis $B(I)$ and related results.
\vspace{0.25 cm}

\section{Proofs}\label{k-wise proof section}
We first recall for later reference several well-known identities involving binomial coefficients (see \cite[Lemma 1.4.2]{Galambos3} and \cite[p.14]{Galambos2}):
\begin{enumerate}
\item \begin{equation}\label{proof sec, equation on the sum of binomial coefficients positive}
\sum_{i=0}^t (-1)^i \binom{s}{i}= (-1)^t\binom{s-1}{t}, \qquad s\geq 1,\, t\geq 0,
\end{equation}
\item
\begin{equation}\label{proof sec, equation on the sum of binomial coefficients equal to zero}
\sum_{i=0}^s (-1)^i \binom{s}{i}= 0, \qquad s\geq 0,  
\end{equation}
\item
\begin{equation}\label{proof sec, equation of elementary combinatorial identity}
\binom{n}{i}\binom{n-i}{j-i}= \binom{j}{i}\binom{n}{j},\qquad 0\leq i\leq j\leq n.
\end{equation}
\end{enumerate}
\subsection{Proof of Theorem \ref{resuls sec, theorem on upper bounds}}
\noindent 1. For $p\in [0,1]$, we have
\begin{equation}\label{proof sec, equation of constraints second interval}
\everymath{\displaystyle}
\setlength\arraycolsep{0.1pt}
\begin{array}{ccccccccccccc}
v_0 & + &\binom{n}{1}v_1 ~~& + & \binom{n}{2}v_2&+~\cdots~& + &\binom{n}{k}v_k & + &~\cdots ~& + &\binom{n}{n}v_n &=1,\\
&&&&&&&&&&&&\\
&&~~~~v_1 & +& \binom{n-1}{1}v_2~~~ &+~\cdots~& + &\binom{n-1}{k-1}v_k~~~ & + &~\cdots~ & + &\binom{n-1}{n-1}v_n &=p,\\
&&&&&& & &&\vdots & & &\\
&&&&&& &~~~~~~~~v_k & + &~\cdots~ & + &\binom{n-k}{n-k}v_n &=p^k.\\
\end{array}
\end{equation}
\tcb{We proceed to multiply the last $(m+1)$ equalities in \eqref{proof sec, equation of constraints second interval} by, respectively,}
$$\binom{n-k+m}{0},\, -\binom{n-k+m}{1},\, \binom{n-k+m}{2}, \, -\binom{n-k+m}{3}, \ldots, \binom{n-k+m}{m}.$$
\tcb{We then add up the resulting equalities, which gives an equality of the form}
\begin{equation}\label{proof sec, equation M(n,k,p) less than second interval}
\begin{split}
c_{k-m}v_{k-m} + c_{k-m +1} v_{k-m +1} +\cdots +  c_k v_k + \cdots + c_n v_n&= \sum_{i=0}^m (-1)^i \binom{n-k+m}{i}\cdot p^{k-m +i}\\
&= p^{k-m}\cdot B(n-k+m,m,p).
\end{split}
\end{equation}
This equality will yield an upper bound on $v_n$ if we find $c_n$ and show that all other coefficients $c_l$ on the left-hand side of \eqref{proof sec, equation M(n,k,p) less than second interval} are non-negative. Using  \eqref{proof sec, equation on the sum of binomial coefficients positive}, \eqref{proof sec, equation on the sum of binomial coefficients equal to zero}, and \eqref{proof sec, equation of elementary combinatorial identity}, we clearly have:
\begin{itemize}
\item 
  \begin{equation}\label{proof sec, equation of coefficient c(k-m)}
  c_{k-m}=1.   
\end{equation}
\item The following $m$ coefficients vanish because
\begin{equation}\label{proof sec, equation2 of ck in the second interval}
\begin{split}
   c_{k-m+j} & =\sum_{i=0}^j (-1)^i \binom{n-k+m}{i}\binom{n-(k-m+i)}{j-i}\\
   &= \sum_{i=0}^j (-1)^i \binom{j}{i}\binom{n-k+m}{j}\\
   &= \binom{n-k+m}{j} \cdot \sum_{i=0}^j (-1)^i \binom{j}{i}=0,\\
\end{split} \qquad j=1,2,\ldots,m.
\end{equation}
\item The final $n-k$ \tcb{coefficients} are strictly positive:
\begin{equation}\label{proof sec, equation3 of ck in the second interval}
\begin{split}
  c_{k+j}&= \sum_{i=0}^m (-1)^i \binom{n-k+m}{i}\binom{n-(k-m+i)}{k+j-(k-m+i)}\\
  &=\sum_{i=0}^m (-1)^i \binom{n-k+m}{i}\binom{n-k+m-i}{j+m-i}\\
  &=\sum_{i=0}^m (-1)^i \binom{j+m}{i}\binom{n-k+m}{j+m}\\
 &=\binom{n-k+m}{j+m} \cdot \sum_{i=0}^m (-1)^i \binom{j+m}{i}\\
 &=\binom{n-k+m}{j+m} \cdot \binom{j+m-1}{m}>0.\\  
\end{split}\qquad j=1,2,\ldots, n-k.
\end{equation}
\end{itemize}
In particular, $c_n=\binom{n-k+m -1}{m}$, and therefore
$$v_n\leq \frac{p^{k-m}}{\binom{n-k+m-1}{m}}\cdot B(n-k+m,m,p).$$
Hence,
$$M(n,k,p)\leq \frac{p^{k-m}}{\binom{n-k+m-1}{m}}\cdot B(n-k+m,m,p).$$
\vspace{0.25 cm}

\noindent 2.
Recall that
\begin{equation}\label{proof sec, equation of recalling definition of M(n,k,p)}
M(n,k,p)= \max_{\P\in A_E(n,k,p)}\P(X_1=\cdots=X_n=1). 
\end{equation}
Let $A_i=\{X_i=1\}$, $1\leq i\leq n$. It follows from \eqref{pre, equation of Si for arbitrary events} and \eqref{pre, equation of  Bonferroni inequalities} that
\begin{equation}\label{proof sec, equation of upper bound of v(n) last interval}
\begin{split}
\P(X_1=\cdots=X_n=1)&\leq \sum_{t=0}^k (-1)^t \sum_{1\leq i_1<i_2<\cdots<i_t\leq n} \P(A_{i_1}^C\cap A_{i_2}^C\cap\cdots\cap A_{i_t}^C)\\
&=\sum_{t=0}^k (-1)^t \binom{n}{t}(1-p)^t= B(n,k,1-p).
\end{split}
\end{equation}
By \eqref{proof sec, equation of recalling definition of M(n,k,p)} and \eqref{proof sec, equation of upper bound of v(n) last interval},
$$M(n,k,p)\leq B(n,k,1-p).$$ This concludes the proof.
\vspace{0.25 cm}

\subsection{Proof of Lemma \ref{prelim, lemma for Bonferroni polynomial}}
\noindent $1.$ We have
\begin{equation}\label{per sec, equation of B(n,k,p) using recurrence realtion}
\begin{split}
B(n,k,p)&= \sum_{i=0}^k (-1)^i\binom{n}{i}p^i =  \sum_{i=0}^k (-1)^i\left(\binom{n-1}{i}+ \binom{n-1}{i-1}\right)p^i\\  
&= \sum_{i=0}^k (-1)^i\binom{n-1}{i}p^i + \sum_{i=0}^k (-1)^i\binom{n-1}{i-1}p^i\\
&= \sum_{i=0}^k (-1)^i\binom{n-1}{i}p^i + \sum_{i=0}^{k-1} (-1)^{i+1}\binom{n-1}{i}p^{i+1}\\
&= B(n-1,k,p)- pB(n-1,k-1,p).
\end{split} 
\end{equation}
Simplifying the right-hand side of \eqref{per sec, equation of B(n,k,p) using recurrence realtion} further, we obtain
\begin{equation}\label{per sec, equation of B(n,k,p) using recurrence relation simplified}
\begin{split}
B(n,k,p)&= B(n-1,k-1,p)+ (-1)^k \binom{n-1}{k}p^k- pB(n-1,k-1,p)\\
&= (1-p)B(n-1,k-1,p) + (-1)^k \binom{n-1}{k}p^k.
\end{split} 
\end{equation}
\noindent $2.$ We have
\begin{equation*}
\begin{split}
\frac{d}{d p} B(n,k,p)&= \frac{d}{d p}\left(\sum_{i=0}^k(-1)^i\binom{n}{i}p^i\right) = \sum_{i=0}^k(-1)^i\binom{n}{i} i p^{i-1}= \sum_{i=1}^k(-1)^i n \binom{n-1}{i-1}p^{i-1}\\
&=-n  \sum_{i=0}^{k-1}(-1)^i \binom{n-1}{i}p^{i}= -n B(n-1,k-1,p).
\end{split}
\end{equation*}

\noindent $3.$ We use induction on $k$. For $k=0$, both sides of \eqref{prilim sec, equation for Bonferroni polynomial in terms of others} are 1. Assume that \eqref{prilim sec, equation for Bonferroni polynomial in terms of others} holds with $k-1$ instead of $k$. Note that both sides of \eqref{prilim sec, equation for Bonferroni polynomial in terms of others} are polynomials of degree at most~$k$ in $p$ and they coincide at $p=0$. Hence, it suffices to prove that the derivatives with respect to $p$ of both sides of \eqref{prilim sec, equation for Bonferroni polynomial in terms of others} coincide, namely that
\begin{equation}\label{proof sec, equality of two derivatives}
 \frac{d}{d p}B(n,k,p)= \frac{d}{d p}\left((1-p)^k- \sum_{i=0}^{k-1}\binom{n-k}{k-i} p^{k-i}B(n-k+i,i,p)\right). 
\end{equation}
By the preceding part, the right-hand side of \eqref{proof sec, equality of two derivatives} is
\begin{equation*}
\begin{split} 
& -k(1-p)^{k-1}- \sum_{i=0}^{k-1}\binom{n-k}{k-i}(k-i) p^{k-1-i}B(n-k+i,i,p)\\
&~~~ + \sum_{i=0}^{k-1}\binom{n-k}{k-i}(n-k+i) p^{k-i}B(n-k+i-1,i-1,p)\\
=& -k(1-p)^{k-1}- \sum_{i=0}^{k-2}\binom{n-k}{k-i}(k-i) p^{k-1-i}B(n-k+i,i,p)- (n-k)B(n-1,k-1,p)\\
&~~~+ \sum_{i=0}^{k-2}\binom{n-k}{k-(i+1)}(n-k+i+1) p^{k-(i+1)}B(n-k+i,i,p)\\
&~~~ + \binom{n-k}{k} (n-k) p^k B(n-k-1,-1,p)\\
=& - \sum_{i=0}^{k-2}\left(\binom{n-k}{k-i}(k-i)- \binom{n-k}{k-1-i}(n-k+i+1)\right)p^{k-1-i} B(n-k+i,i,p)\\
&~~~ -k(1-p)^{k-1} -(n-k)B(n-1,k-1,p) +0\\
=&-k(1-p)^{k-1}- \sum_{i=0}^{k-2}(-k)\binom{n-k}{k-1-i} p^{k-1-i}B(n-k+i,i,p)\\
&~~~ -(n-k)B(n-1,k-1,p)\\
=&-k\left((1-p)^{k-1}- \sum_{i=0}^{k-2}\binom{n-k}{k-1-i} p^{k-1-i}B(n-k+i,i,p)\right)\\
&~~~ - (n-k)B(n-1,k-1,p)\\
\end{split}   
\end{equation*}
By the induction hypothesis, the last expression reduces to 
$$-kB(n-1,k-1,p)- (n-k)B(n-1,k-1,p)= -nB(n-1,k-1,p).$$
In view of part 2 of the lemma, this proves \eqref{proof sec, equality of two derivatives}.
\vspace{0.25 cm}

\vspace{0.25 cm}

\subsection{Proof of Proposition \ref{result, proposition for B(n,l,p)}}
It follows from \cite[Lemma 1.3.2]{Galambos3} (see also \cite[p.562]{Galambos1}) that, for even $l$,
\begin{equation}\label{proof sec, equation of Bonferroni for even l}
 \binom{m}{1}p - \binom{m}{2}p^2 + \binom{m}{3}p^3 - \cdots -  \binom{m}{l}p^l \leq 1- (1-p)^m.   
\end{equation}
Thus, for even $l$, we have $B(m,l,p)\geq (1-p)^m$ for all $0\leq p\leq 1$. It remains to prove the proposition for odd $l$. We first deal with a few simple cases.
\begin{description}
\item {1.} If $l=1$, then $B(m,l,p)=1-mp$, which is trivially non-negative for $p\leq \frac{1}{m-1+1}=\frac{1}{m}$, and vanishes only at $p=\frac{1}{m}$.
\item {2.} If $l=m-1$, then $B(m,l,p)= (1-p)^m-p^m$. Since $p\leq \frac{1}{2}$, this is certainly nonnegative, vanishing only at $p=1/2$.
\item {3.} If $l=m$, then $B(m,l,p)=(1-p)^m$, which is trivially non-negative, vanishing only at $p=1$.
\end{description}
Now, let $l$ be odd, $3\leq l\leq m-2$. Let
$$p_1= \frac{1}{m-k+1}.$$
It follows from Lemma \ref{prelim, lemma for Bonferroni polynomial}$.2$ (since $l-1$ is even) that $\frac{d}{d p} B(m,l,p)$ is negative, and hence it suffices to check that the required inequality is satisfied and is strict at the right endpoint of the interval, namely that
$$1- \binom{m}{1}p_1 + \binom{m}{2}p_1^2- \cdots- \binom{m}{l}p_1^l> 0.$$
Since $l-1$ is even, we have by \cite[Lemma 1.3.2]{Galambos3}
\begin{equation*}
 \begin{split}
\binom{m}{1}p_1 -  \binom{m}{2}p_1^2 + \binom{m}{3}p_1^3 - \cdots -  \binom{m}{l-1}p_1^{l-1} &\leq 1- (1-p_1)^m,   
 \end{split}   
\end{equation*}
which yields
\begin{equation}\label{proof sec, equation of Bonferroni for even l-1}
\begin{split}
1-\binom{m}{1}p_1 +  \binom{m}{2}p_1^2 - \binom{m}{3}p_1^3 + \cdots +  \binom{m}{l-1}p_1^{l-1} -\binom{m}{l}p_1^{l}& \geq(1-p_1)^m-\binom{m}{l}p_1^{l}.\\
\end{split}  
\end{equation}
Thus, it suffices to show that for odd $l$
$$\binom{m}{l}p_1^l < (1-p_1)^m,$$
which is equivalent to 
\begin{equation}\label{proof sec, first equation right-hand side less than e}
\frac{m\cdot(m-1)\cdots (m-l+1)}{l!(m-l)^l} < \left(\frac{m-l}{m-l+1}\right)^{m-l}.    
\end{equation}
The right-hand side of \eqref{proof sec, first equation right-hand side less than e} is at least $1/e$, so that it suffices to show that
\begin{equation}\label{proof sec, for odd l equation less than e}
 \frac{m-l+1}{1\cdot (m-l)}\cdot \frac{m-l+2}{2\cdot (m-l)}\cdot \frac{m-l+3}{3\cdot (m-l)}\cdot\cdots\cdot \frac{m}{l\cdot (m-l)}= \prod_{j=1}^l\left(\frac{1}{j}+ \frac{1}{m-l}\right)< \frac{1}{e}.   
\end{equation}
We show this by distinguishing between several cases:
\begin{description}
\item {1.} $l\geq 7$: We have
$$\prod_{j=1}^l\left(\frac{1}{j}+ \frac{1}{m-l}\right)\leq \prod_{j=1}^7\left(\frac{1}{j}+ \frac{1}{2}\right)= \frac{9}{32}< \frac{1}{e}.$$
\item {2.} $l=3,5$ and $m\geq 10$: For every fixed $l$, as $m$ increases the product in \eqref{proof sec, for odd l equation less than e} consists of the same number of factors, but they become smaller and smaller. Therefore, if~\eqref{proof sec, for odd l equation less than e} holds for some $(m,l)$, then it must hold for $(m+1,l),(m+2,l),\ldots$. We verified ~\eqref{proof sec, for odd l equation less than e} numerically for $m=10$, and $l=3$ and $l=5$; hence it must be true for $m\geq 10$.
\item {3.} $l=3,5$ and $m<10$: \tcb{It is easily verified} that \eqref{result sec, positivity of B(n,l,p) in first half} holds (with a strict inequality) \tcb{via straightforward calculation.}
\end{description}
\vspace{0.25 cm}

\subsection{Proof of Theorem \ref{resuls sec, theorem on first half}}
Suppose first that $M(n,k,p)\equiv v_n=p^k$. By \eqref{pre, sec, equation sum of vi equal to pi in the first form}, all the $v_i$, $k\leq i \leq n-1$, must vanish. We need to show that
\begin{equation*}
\begin{split}
v_i& = (1-p)\cdot p^i B(n-i-1,k-i-1,p), \qquad 0\leq i\leq k-1.
\end{split}
\end{equation*}
Using the generalized inclusion-exclusion principle (see \cite[Corollary~5.2]{Aigner}) and the fact that $\P$ is an exchangeable measure, we have
\begin{equation}\label{proof sec, equation of v(i) in thm M(n,k,p)=p^k}
v_i= \sum_{j=i}^n (-1)^{j-i}\binom{n-i}{j-i}\P(X_1=X_2=\cdots =X_j=1),\qquad i=0,1,\ldots,n.    
\end{equation}
Since the $X_i$-s are $k$-wise independent,
\begin{equation}\label{proof sec, equation of w(j) in thm M(n,k,p)=p^k}
\P(X_1=X_2=\cdots =X_j=1) =p^j,\qquad j=0,1,\ldots,k.   
\end{equation}
For $j\geq k$,
\begin{equation*}
=p^k=\P(X_1=\cdots =X_k=1)\geq \P(X_1=\cdots =X_j=1)\geq \P(X_1=\cdots =X_n=1)=p^k  
\end{equation*}
so that
\begin{equation}\label{proof sec, equation of w(J) in thm M(n,k,p)=p^k for j>k}
\P(X_1=X_2=\cdots =X_j=1) =p^k,\qquad j=k+1,k+2,\ldots,n.   
\end{equation}
By \eqref{proof sec, equation of v(i) in thm M(n,k,p)=p^k}, \eqref{proof sec, equation of w(j) in thm M(n,k,p)=p^k}, and \eqref{proof sec, equation of w(J) in thm M(n,k,p)=p^k for j>k}, we have for $0\leq i<k$
\begin{equation}\label{proof sec, formula of vi in the case vn pk}
\begin{split}
v_i& =  p^i- (n-i)p^{i+1} + \binom{n-i}{2}p^{i+2}- \binom{n-i}{3}p^{i+3}+ \cdots\\
&~~~  + (-1)^{k-i-1}\binom{n-i}{k-i-1}p^{k-1} + (-1)^{k-i}\binom{n-i}{k-i}p^k + (-1)^{k-i+1}\binom{n-i}{k-i+1}p^k + \cdots\\
& ~~~+ (-1)^{n-i}\binom{n-i}{n-i}p^k\\
&= p^i- (n-i)p^{i+1} + \binom{n-i}{2}p^{i+2}- \binom{n-i}{3}p^{i+3}+ \cdots\\
&~~~  + (-1)^{k-1-i}\binom{n-i}{k-i-1}p^{k-1} + (-1)^{k-i}\binom{n-i-1}{k-i-1}p^k,\\
\end{split}
\end{equation}
\tcb{where in the last step we have relied} on the identity
$$\binom{m}{j}- \binom{m}{j+1} + \binom{m}{j+2}- \cdots + (-1)^{m-j}\binom{m}{m}= \binom{m-1}{j-1}.$$
By \eqref{proof sec, defi of bonferroni}, \eqref{proof sec, formula of vi in the case vn pk}, and Lemma \ref{prelim, lemma for Bonferroni polynomial}$.1$,
\begin{equation}\label{proof sec, equation of vi in terms of nonferroni}
\begin{split}
v_i&=p^i B(n-i,k-i-1,p)+ (-1)^{k-i}\binom{n-i-1}{k-i-1}p^k\\ 
&= p^i(B(n-i-1,k-i-1,p)- p B(n-i-1,k-i-2,p))\\
&~~~ - p^{i+1}(B(n-i-1,k-i-1,p)-B(n-i-1,k-i-2,p))\\
&= p^i (1-p) B(n-i-1,k-i-1,p).
\end{split}   
\end{equation}
In particular,
\begin{equation*}
0\leq v_{k-2}= p^{k-2}(1-p)B(n-k+1,1,p)= p^{k-2}(1-p)(1- (n-k+1)p),
\end{equation*}
which implies that the condition $p\leq \frac{1}{n-k+1}$ is necessary.

It follows from Proposition \ref{result, proposition for B(n,l,p)} that the solution in \eqref{result ses, equation of optimal distribution in the first interval} is a feasible solution for $p\leq \frac{1}{n-k+1}$. Using Theorem \ref{resuls sec, theorem on upper bounds}$.1$, we conclude that the solution is optimal. Hence the condition $p\leq \frac{1}{n-k+1}$ is sufficient.
\vspace{0.25 cm}

\subsection{Proof of Theorem \ref{resuls sec, theorem on last half}}
By the generalized inclusion-exclusion principle (see \cite[Corollary~5.2]{Aigner}) and the fact that $\P$ is exchangeable, 
\begin{subequations}
\begin{align}\label{proof sec, equation for constraints (a)}
&\sum_{j=i}^n\binom{n-k}{j-i}v_j= p^i(1-p)^{k-i},&\quad 0\leq i\leq k,\\ \label{proof sec, equation for constraints (b)}
&v_i\geq 0,& \qquad 0\leq i\leq n.
\end{align}
\end{subequations}
Since $k$ is even, it follows from Theorem \ref{resuls sec, theorem on upper bounds}$.2$ that, for every feasible solution,
\begin{equation}\label{proof sec, equation on Vn less than B(n,k,p)}
\begin{split}
\P(X_1= X_2=\cdots=X_n=1)&\leq B(n,k,1-p).
\end{split}
\end{equation}
\tcb{Now moving in the opposite direction}, we have to find a feasible solution of the system of constraints \eqref{proof sec, equation for constraints (a)}-\eqref{proof sec, equation for constraints (b)}
such that $v_n = B(n,k,1-p)$. We choose \begin{align}\label{proof sec, equation of optimal distribution v(i) in the last interval}
 v_i= \begin{cases}
     0,\quad &0\leq i\leq n-k-1,\\
     (1-p)^{n-i}\cdot B(i ,i-(n-k),1-p),\quad &n-k\leq i\leq n.\\
 \end{cases}    
\end{align}
By Proposition \ref{result, proposition for B(n,l,p)}, the $v_i$-s in \eqref{proof sec, equation of optimal distribution v(i) in the last interval} are non-negative for $p\geq1-\frac{1}{n-k+1}$.
We rewrite \eqref{proof sec, equation for constraints (a)} in the form
$$\sum_{j=0}^k\binom{n-k}{n-k-i+j}v_{n-k+j}= p^i(1-p)^{k-i}, \qquad 0\leq i\leq k.$$
Equivalently,
\begin{equation}\label{proof sec, equation of the formal v(n-k+j)}
\sum_{j=0}^i\binom{n-k}{i-j}v_{n-k+j}= p^i (1-p)^{k-i},\qquad 0\leq i\leq k.    
\end{equation}
By Lemma \ref{prelim, lemma for Bonferroni polynomial}$.3$ and \eqref{proof sec, equation of optimal distribution v(i) in the last interval}, we have that for $0\leq i\leq k$
\begin{equation*}
\begin{split}
\sum_{j=0}^i \binom{n-k}{i-j}v_{n-k+j}&=\sum_{j=0}^i \binom{n-k}{i-j} (1-p)^{k-j}B(n-k+j,j,1-p)\\
&= (1-p)^{k-i} \sum_{j=0}^i \binom{n-k}{i-j} (1-p)^{i-j}B(n-k+j,j,1-p)\\
&=(1-p)^{k-i}p^i.
\end{split}   
\end{equation*}
Hence, the $v_i$-s in \eqref{proof sec, equation of optimal distribution v(i) in the last interval} form the optimal distribution. This completes the proof.
\vspace{0.25 cm}

\subsection{Proof of Proposition \ref{prop for M(n,2,p) except finitely many p}}
Let $I=\{j-1,j,n\}$ for some $1\leq j\leq n-1$. The inverse $B(I)^{-1}=[b_{s,t}]_{s\in I}^{t=0,1,2}$ is easily calculated as
{ \everymath={\displaystyle}
\begin{equation}\label{proof sec, equation of inverse of B(I) for M(n,2,p)}
B(I)^{-1}=\begin{bmatrix}
 \frac{n j}{n-j+1} & -\frac{n+j-1}{n-j+1} & \frac{2}{n-j+1}\\
 &&\\
 -\frac{n(j-1)}{n-j} & \frac{n+j-2}{n-j} & -\frac{2}{n-j}\\
 &&\\
 \frac{j(j-1)}{(n-j+1)(n-j)} & -\frac{2(j-1)}{(n-j+1)(n-j)} & -\frac{2}{(n-j+1)(n-j)}
\end{bmatrix}.    
\end{equation}
}
It follows from the discussion in Subsection \ref{subsection on LP} that
\begin{equation}\label{proof sec, equation of v(i) for M(n,2,k)}
v_i= \frac{1}{\binom{n}{i}}\sum_{l=0}^2 b_{i,l}\binom{n}{l}p^l,\qquad i=j-1,j,n.    
\end{equation}
By \eqref{proof sec, equation of inverse of B(I) for M(n,2,p)} and \eqref{proof sec, equation of v(i) for M(n,2,k)},
\begin{equation}\label{proof sec, equation of v(j-1) for M(n,2,p)}
\begin{split}
v_{j-1}&= \frac{1}{\binom{n}{j-1}}\cdot \frac{n}{n-j+1}\cdot (1-p)\cdot (j+p-np)\\ 
&= \frac{1}{\binom{n-1}{j-1}}\cdot (1-p)\cdot (j+p-np),
\end{split}    
\end{equation}
\begin{equation}\label{proof sec, equation of v(j) for M(n,2,p)}
\begin{split}
v_{j}&= \frac{1}{\binom{n}{j}}\cdot \frac{n}{n-j}\cdot (1-p)\cdot (np-p-j+1)\\ 
&= \frac{1}{\binom{n-1}{j}} \cdot (1-p)\cdot (np-p-j+1),
\end{split}    
\end{equation}
and 
\begin{equation}\label{proof sec, equation of v(n) for M(n,2,p)}
\begin{split}
v_{n}&= \frac{1}{(n-j)(n-j+1)}\cdot (n(n-1)p^2 -2n(j-1)p + j(j-1))\\ 
&= \frac{1}{\binom{n-j+1}{2}}\left(\binom{n}{2}p^2 -n(j-1)p + \binom{j}{2}\right).
\end{split}    
\end{equation}
By \eqref{proof sec, equation of v(j-1) for M(n,2,p)}-\eqref{proof sec, equation of v(n) for M(n,2,p)}, considered as functions of $p$ on the interval $\left[\frac{j-1}{n-1},\frac{j}{n-1}\right]$:
\begin{itemize}
\item $v_{j-1}$ is decreasing and vanishes at the right endpoint;
\item $v_{j}$ vanishes at the left endpoint and is increasing;
\item $v_n$ attains its minimum, which is positive, at the left endpoint.
\end{itemize}
It then follows from Theorem \ref{pre sec, theorem on B(I) tight bound} that
$$M(n,2,p)=\frac{1}{\binom{n-j+1}{2}}\left(\binom{n}{2}p^2 -n(j-1)p + \binom{j}{2}\right), \qquad p\in \left[\frac{j-1}{n-1},\frac{j}{n-1}\right],$$
and the distribution is given by \eqref{proof sec, equation of v(j-1) for M(n,2,p)}-\eqref{proof sec, equation of v(n) for M(n,2,p)}.
\vspace{0.25 cm}

\subsection{Proof of Theorem \ref{resuls sec, theorem on second interval 1}}
\noindent 1. By Theorem \ref{resuls sec, theorem on first half}, the optimal distribution for $0\leq p \leq p_1$ is given by
\begin{equation}\label{proof sec, equation of vi in the first interval for proof of the second interval}
\begin{split}
v_i= \begin{cases}(1-p)\cdot p^i\cdot B(n-i-1,k-i-1,p),& \quad 0\leq i\leq k-1,\\
0,& \quad k\leq i\leq n-1, \\
p^k,&\quad i=n.    
\end{cases}     
\end{split}
\end{equation}
In particular, at the point $p=p_1$,
$$v_{k-2}= (1-p_1)p_1^{k-2}B(n-k+1,1,p_1)=(1-p_1)p_1^{k-2}\left(1-(n-k+1)\cdot \frac{1}{n-k+1}\right)=0,$$
while from Proposition \ref{result, proposition for B(n,l,p)} it follows that
$$v_i>0,\qquad i\in \{0,1,\ldots, k-3,k-1,n\}.$$
The optimal distribution changes continuously with $p$, and therefore
\begin{equation}\label{proof sec, equation for vi strictly positive}
v_i>0,\qquad i\in \{0,1,\ldots, k-3,k-1,n\}, \,\, p\in [p_1,p_2],  
\end{equation}
for some $p_2>p_1$.
Let $I=\{0,1,2,\ldots, k-3,k-1, k, n\}$. It follows from the discussion in Subsection \ref{subsection on LP} and \eqref{pre sec, equation of elements of matrix B} that
\begin{equation}\label{proof sec, equation of vi second interval}
 v_i=\frac{1}{\binom{n}{i}}\cdot \sum_{l=0}^k b_{i,l} \binom{n}{l}p^l, \qquad    i\in I,
\end{equation}
where 
\begin{equation}\label{proof sec, equation of b(st) second interval}
b_{s,t} =  \frac{(-1)^t \sum_{\alpha=0}^t (-1)^\alpha \binom{t}{\alpha}\prod_{u\in I-\{s\}}(u- \alpha)}{\prod_{u\in I-\{s\}}(u- s)},\qquad s\in I,\, 0\leq t\leq k.   
\end{equation}
By \eqref{proof sec, equation of b(st) second interval},
\begin{equation}\label{proof sec, equation of b(0,j) second interval for i less then k-3}
\begin{split}
b_{0,j}=\begin{dcases}
 (-1)^j,& \qquad 0\leq j\leq k-3, \\
 (-1)^{k-2} + (-1)^{k-1}C_0,& \qquad j=k-2,\\
(-1)^{k-1} - (-1)^{k-1}\binom{k-1}{k-2}C_0,&\qquad j=k-1,\\
(-1)^{k} + (-1)^{k-1}\binom{k}{k-2}C_0,&\qquad j=k,
\end{dcases}
\end{split}  
\end{equation}
where
\begin{equation}\label{proof sec, equation of C0}
\begin{split}
C_0&= \frac{(k-2)!\cdot 2\cdot (n-k+2)}{k!\cdot n} = \frac{\binom{n-1}{k}}{\binom{n-k+1}{2}\cdot \binom{n}{k-2}}.
\end{split}    
\end{equation}
In view of \eqref{proof sec, equation of vi second interval}, \eqref{proof sec, equation of b(0,j) second interval for i less then k-3}, and \eqref{proof sec, equation of C0}, we obtain
\begin{equation}\label{proof sec, equation of v(0) in second interval}
\begin{split}
v_0&= \sum_{j=0}^{k}b_{0,j} \binom{n}{j} p^j\\
&= \sum_{j=0}^{k}(-1)^j\binom{n}{j}p^j + \\
&~~~ (-1)^{k-1}\left(\binom{n}{k-2}p^{k-2} - \binom{k-1}{k-2}\binom{n}{k-1} p^{k-1} + \binom{k}{k-2}\binom{n}{k}p^{k-2}\right) C_0\\
&=B(n,k,p)+ (-1)^{k-1}\binom{n}{k-2}\cdot p^{k-2}B(n-k+2,2,p) C_0\\
&=B(n,k,p) +(-1)^{k-1} \frac{\binom{n-1}{k}}{\binom{n-k+1}{2}}\cdot p^{k-2} B(n-k+2,2,p).\\
\end{split}
\end{equation}
By \eqref{proof sec, equation of b(st) second interval}, for $1\leq i\leq k-3$,
\begin{equation}\label{proof sec, equation of b(i,j) second interval for i less then k-3}
\begin{split}
b_{i,j}=\begin{dcases}
(-1)^j(-1)^i\binom{j}{i},& \qquad i\leq j\leq k-3, \\
(-1)^{k-2}(-1)^i\binom{k-2}{i} + (-1)^{k-i-1}C_{i}, & \qquad j=k-2,\\
(-1)^{k-1}(-1)^i\binom{k-1}{i} - (-1)^{k-i-1}\binom{k-1}{k-2}C_{i}, & \qquad j=k-1,\\
(-1)^{k}(-1)^i\binom{k}{i} + (-1)^{k-i-1}\binom{k}{k-2}C_{i},& \qquad j=k,
\end{dcases}
\end{split}   
\end{equation}
where 
\begin{equation}\label{proof sec, equation of C1}
\begin{split}
C_{i}&= \frac{(k-2)!\cdot 2\cdot (n-k+2)}{i!\cdot (k-i)!\cdot (n-i)}= \frac{\binom{n}{i}\cdot \binom{n-i-1}{k-i}}{\binom{n-k+1}{2}\cdot \binom{n}{k-2}}.
\end{split}  \qquad 1\leq i\leq k-3.  
\end{equation}
(Here and below, we hint implicitly that all unspecified coefficients vanish. For example, in \eqref{proof sec, equation of b(i,j) second interval for i less then k-3} we have $b_{i,j}=0$ for $j<i$.)

By \eqref{proof sec, equation of vi second interval}, \eqref{proof sec, equation of b(i,j) second interval for i less then k-3}, and \eqref{proof sec, equation of C1}, for $1\leq i\leq k-3$,
\begin{equation*}
\begin{split}
\binom{n}{i}v_i&= \sum_{j=0}^{k}b_{i,j} \binom{n}{j} p^j\\
&= \sum_{j=i}^{k}(-1)^j(-1)^i \binom{j}{i}\binom{n}{j}p^j\\
&~~~ + (-1)^{k-i-1} p^{k-2} \left(\binom{n}{k-2} - \binom{k-1}{k-2}\binom{n}{k-1} p + \binom{k}{k-2}\binom{n}{k}p^2\right) C_{i}\\
&= \binom{n}{i} p^iB(n-i,k-i,p) +(-1)^{k-i-1} p^{k-2} \binom{n}{k-2} B(n-k+2,2,p) C_{i},\\
\end{split}
\end{equation*}
and therefore
\begin{equation}\label{proof sec, equation of v(i) in second interval i <=k-3}
v_i=p^iB(n-i,k-i,p) +(-1)^{k-i-1} p^{k-2} \frac{\binom{n-i-1}{k-i}}{\binom{n-k+1}{2}} B(n-k+2,2,p).    
\end{equation}
By \eqref{proof sec, equation of b(st) second interval},
\begin{equation}\label{proof sec, equation of b(k-1,i) second interval}
\begin{split}
b_{k-1,j}= \begin{dcases}
C_{k-1}, & \qquad j=k-2,\\
- \binom{k-1}{k-2 }C_{k-1} +1, & \qquad j=k-1,\\
\binom{k}{k-2 }C_{k-1} -\binom{k}{k-1},& \qquad j=k, \\    
\end{dcases}
\end{split}
\end{equation}
where 
\begin{equation}\label{proof sec, equation of C2}
\begin{split}
C_{k-1}&= \frac{(k-2)! \cdot 2 \cdot (n-k+2)}{(k-1)!\cdot (n-k+1)}= \frac{\binom{n}{k-1}}{\binom{n}{k-2}}\cdot \frac{\binom{n-k}{1}}{\binom{n-k+1}{2}}.
\end{split}    
\end{equation}
By \eqref{proof sec, equation of vi second interval}, \eqref{proof sec, equation of b(k-1,i) second interval}, and \eqref{proof sec, equation of C2}, we have
\begin{equation*}
\begin{split}
 \binom{n}{k-1}v_{k-1} &= b_{k-1,k-2} \binom{n}{k-2}p^{k-2} + b_{k-1,k-1} \binom{n}{k-1}p^{k-1}  + b_{k-1,k} \binom{n}{k}p^{k} \\
 &= C_{k-1}\cdot p^{k-2} \binom{n}{k-2} - \binom{k-1}{k-2}\binom{n}{k-1} C_{k-1}\cdot p^{k-1} + \binom{n}{k-1}p^{k-1}\\
 &~~~ + \binom{k}{k-2}\binom{n}{k} C_{k-1} \cdot p^{k-1}- \binom{k}{k-1}\binom{n}{k}p^k\\
 &= \binom{n}{k-1} p^{k-1}B(n-k+1,1,p) + \binom{n}{k-2}p^{k-2}B(n-k+2,2,p )C_{k-1}\\
 &= \binom{n}{k-1} p^{k-1}B(n-k+1,1,p) + \frac{\binom{n}{k-1}\cdot \binom{n-k}{1}}{\binom{n-k+1}{2}}p^{k-2}B(n-k+2,2,p ),\\
\end{split}   
\end{equation*}
and therefore
\begin{equation}\label{proof sec, equation of v(k-1) in second interval}
v_{k-1} =p^{k-1}B(n-k+1,1,p) + \frac{\binom{n-k}{1}}{\binom{n-k+1}{2}}p^{k-2}B(n-k+2,2,p ).   \end{equation}
By \eqref{proof sec, equation of b(st) second interval},
\begin{equation}\label{proof sec, equation of b(kj) second interval}
\begin{split}
b_{k,j}= \begin{dcases}
-C_k,&\qquad j=k-2,\\
\binom{k-1}{k-2} C_k,&\qquad j=k-1,\\
-\binom{k}{k-2} C_k +1,&\qquad j=k,\\
\end{dcases}
\end{split}
\end{equation}
where
\begin{equation}\label{proof sec, equation of C3}
\begin{split}
C_k&= \frac{(k-2)! \cdot 2 \cdot (n-k+2)}{k!\cdot (n-k)}=\frac{\binom{n}{k}}{\binom{n}{k-2}}\cdot \frac{1}{\binom{n-k+1}{2}}.
\end{split}   
\end{equation}
In view of \eqref{proof sec, equation of vi second interval}, \eqref{proof sec, equation of b(kj) second interval}, and \eqref{proof sec, equation of C3}, we obtain
\begin{equation*}
\begin{split}
\binom{n}{k}v_k&= b_{k,k-2} p^{k-2}\binom{n}{k-2} + b_{k,k-1} p^{k-1}\binom{n}{k-1}+ b_{k,k} p^{k}\binom{n}{k}\\
&= -C_k p^{k-2}\binom{n}{k-2} + \binom{k-1}{k-2} C_k p^{k-1}\binom{n}{k-1}- \binom{k}{k-2} C_k p^{k}\binom{n}{k} + p^{k}\binom{n}{k}\\
&=\binom{n}{k}p^k -p^{k-2}\binom{n}{k-2}B(n-k+2,2,p)C_k\\
&=\binom{n}{k} p^k -p^{k-2}\cdot \frac{\binom{n}{k}}{\binom{n-k+1}{2}}B(n-k+2,2,p),
\end{split}
\end{equation*}
and therefore
\begin{equation}\label{proof sec, equation of v(k) in second interval}
v_k =p^k -p^{k-2}\cdot \frac{1}{\binom{n-k+1}{2}}B(n-k+2,2,p).  
\end{equation}
Using Lemma \ref{prelim, lemma for Bonferroni polynomial}$.1$, we further simplify \eqref{proof sec, equation of v(k) in second interval} as
\begin{equation}\label{proof sec, equation of v(k) in second interval simplified}
\begin{split}
v_k&= p^k -p^{k-2}\frac{1-p}{\binom{n-k+1}{2}}B(n-k+1,1.p) - p^k\\
&= -p^{k-2}\frac{1-p}{\binom{n-k+1}{2}}B(n-k+1,1,p)\\
 &=\frac{1-p}{\binom{n-k+1}{2}}\cdot \left(-1+(n-k+1)p\right).\\
\end{split}
\end{equation}
By \eqref{proof sec, equation of v(k) in second interval simplified}, we observe that $v_k=0$ at the point $p_1$ while $v_k>0$ for $p>p_1$.
Using \eqref{proof sec, equation of vi second interval}, we obtain
\begin{equation}\label{proof sec, equation of v(n) in second interval}
v_n= \sum_{l=0}^k b_{n,l} \binom{n}{l}p^l.    
\end{equation}
By \eqref{proof sec, equation of b(st) second interval}, 
\begin{equation}\label{proof sec, equation of b(n0) second interval}
\begin{split}
b_{n,j}=\begin{dcases}
C_n,&\qquad j=k-2,\\
-\binom{k-1}{k-2}C_n, &\qquad j=k-1,\\
\binom{k}{k-2}C_n,&\qquad j=k,\\    
\end{dcases}
\end{split}
\end{equation}
where
\begin{equation}\label{proof sec, equation of C4}
\begin{split}
C_n&= \frac{1}{\binom{n}{k-2}}\cdot\frac{1}{\binom{n-k+1}{2}}.
\end{split}   
\end{equation}
In view of \eqref{proof sec, equation of v(n) in second interval}, \eqref{proof sec, equation of b(n0) second interval}, \eqref{proof sec, equation of C4}, we conclude that
\begin{equation}\label{proof sec, equation2 of v(n) in second interval}
\begin{split}
v_n&= b_{n,k-2 }\binom{n}{k-2}p^{k-2} + b_{n,k-1 }\binom{n}{k-1} p^{k-1}+ b_{n,k}\binom{n}{k}p^{k}\\  
&=\left(\binom{n}{k-2} -\binom{k-1}{k-2}\binom{n}{k-2} + \binom{k}{k-2}\binom{n}{k}\right)C_n\\
&= p^{k-2}\binom{n}{k-2}B(n-k+2,2,p)C_n\\
&= \frac{1}{\binom{n-k+1}{2}}\cdot p^{k-2}\cdot B(n-k+2,2,p).
\end{split}
\end{equation}
Using Theorem \ref{resuls sec, theorem on upper bounds}$.1$, \eqref{proof sec, equation for vi strictly positive}, \eqref{proof sec, equation of v(k) in second interval simplified}, and \eqref{proof sec, equation2 of v(n) in second interval}, we finally arrive at
$$M(n,k,p)= \frac{1}{\binom{n-k+1}{2}}\cdot p^{k-2}\cdot B(n-k+2,2,p),\qquad p\in \left[p_1,p_2\right].$$
\vspace{0.25 cm}

\noindent 2. For $k=2,3$, the claim follows readily from Propositions \ref{prop for M(n,2,p) except finitely many p}-\ref{prop for M(n,3,p) except finitely many p}. Thus, let $k\geq 4$.
Using Lemma \ref{prelim, lemma for Bonferroni polynomial}$.1$ and \eqref{proof sec, equation of v(k-1) in second interval}, we obtain 
\begin{equation}\label{proof sec, equation of vk-1 for k in second interval}
\begin{split}
v_{k-1}&= p^{k-1}B(n-k+1,1,p) + \frac{\binom{n-k}{1}}{\binom{n-k+1}{2}}\cdot p^{k-2} B(n-k+2,2,p)\\
&= p^{k-1}\left((1-p)B(n-k,0,p)- \binom{n-k}{1}p\right) \\
&~~~+ \frac{2}{n-k+1}\cdot p^{k-2} \left((1-p)B(n-k+1,1,p) +\binom{n-k+1}{2}p^2\right)\\
&= p^{k-1}(1-p) +\frac{2}{n-k+1}\cdot p^{k-2}(1-p)B(n-k+1,1,p)\\
&= \frac{1-p}{n-k+1}\cdot p^{k-2}\cdot \left(2-p(n-k+1)\right)=\frac{1-p}{n-k+1}\cdot  p^{k-2}\cdot f_{k-1}(p),\\
\end{split}    
\end{equation}
where $f_{k-1}(p)= 2- (n-k+1)p$.
Using Lemma \ref{prelim, lemma for Bonferroni polynomial}$.1$ and \eqref{proof sec, equation of v(i) in second interval i <=k-3}, we obtain
\begin{equation}\label{proof sec, equation of vk-3 for k in second interval}
\begin{split}
v_{k-3}&= p^{k-3}B(n-k+3,3,p) + \frac{\binom{n-k+2}{3}}{\binom{n-k+1}{2}}\cdot p^{k-2} B(n-k+2,2,p)\\
&= p^{k-3}\left((1-p)B(n-k+2,2,p)- \binom{n-k+2}{3}p^3\right)\\
&~~~ + \frac{\binom{n-k+2}{3}}{\binom{n-k+1}{2}}\cdot p^{k-2} \left((1-p)B(n-k+1,1,p)+ \binom{n-k+1}{2}p^2\right)\\
&=p^{k-3}(1-p)\left(B(n-k+2,2,p) + \frac{n-k+2}{3}\cdot p\cdot B(n-k+1,1,p)\right)\\
&=(1-p) \cdot p^{k-3}\cdot \left(1-\frac{2}{3}(n-k+2)p + \frac{1}{6} (n-k+1)(n-k+2)p^2\right),\\
&=(1-p) \cdot p^{k-3}\cdot f_{k-3}(p),\\
\end{split}    
\end{equation}
where 
\begin{equation}\label{proof sec, equation of f(2,k-3)}
f_{k-3}(p)= 1-\frac{2}{3}\binom{n-k+2}{1}p + \frac{1}{3} \binom{n-k+2}{2}p^2.    
\end{equation}
Using Lemma \ref{prelim, lemma for Bonferroni polynomial}$.1$ and \eqref{proof sec, equation of v(i) in second interval i <=k-3}, we obtain
\begin{equation}\label{proof sec, equation of vk-4 for k in second interval}
\begin{split}
v_{k-4}&= p^{k-4}B(n-k+4,4,p) - \frac{\binom{n-k+3}{4}}{\binom{n-k+1}{2}}\cdot p^{k-2} B(n-k+2,2,p)\\
&= p^{k-4}\left((1-p)B(n-k+3,3,p) +\binom{n-k+3}{4}p^4 \right)\\
&~~~ - \frac{\binom{n-k+3}{4}}{\binom{n-k+1}{2}}\cdot p^{k-2} \left((1-p)B(n-k+1,1,p) +\binom{n-k+1}{2}p^2\right)\\
&= p^{k-4}(1-p)\left(B(n-k+3,3,p)- \frac{\binom{n-k+3}{4}}{\binom{n-k+1}{2}}\cdot p^2\cdot B(n-k+1,1,p) \right)\\
&= (1-p) \cdot p^{k-4}\cdot f_{k-4}(p),\\
\end{split}    
\end{equation}
where 
\begin{equation}\label{proof sec, equation of f(2,k-4)}
f_{k-4}(p)= 1-\binom{n-k+3}{1}p + \frac{5}{6} \binom{n-k+3}{2}p^2- \frac{1}{2}\binom{n-k+3}{3}p^3.    
\end{equation}
By \eqref{proof sec, equation of vk-1 for k in second interval}-\eqref{proof sec, equation of f(2,k-4)}, considered as functions of $p$ on the interval $\left[\frac{1}{n-k+1},\frac{2}{n-k+1}\right]$:
\begin{itemize}
\item $v_{k-1}$ is a decreasing function and vanishes at the right endpoint;
\item $v_{k-3}$ attains its minimum, which is positive, at the right endpoint;
\item $v_{k-4}$ is strictly positive at the left endpoint and strictly negative at the right endpoint. The derivative of $f_{k-4}$ is strictly negative in the interval. Hence, $v_{k-4}$ vanishes exactly once in the interval, at the point $p_2$;
\item $v_i$ is non-negative throughout the sub-interval $[p_1,p_2]$ for $i=0,1,2,\ldots, k-5$ by Theorem~\ref{pre sec, theorem on basis change}.
\end{itemize}
Using Theorem \ref{pre sec, theorem on B(I) tight bound}, we conclude the proof.
\vspace{0.25 cm}

\subsection{Proof of Theorem \ref{resuls sec, theorem on second last interval 1}}
By Theorem \ref{resuls sec, theorem on last half}, the optimal distribution for $[\overline{p}_1,1]$ is given by
\begin{align}\label{proof sec, equation of optimal distribution in the last interval}
 v_i= \begin{cases}
     0,&\quad 0\leq i\leq n-k-1,\\
     (1-p)^{n-i}\cdot B(i ,i-(n-k),1-p),&\quad n-k\leq i\leq n.\\
 \end{cases}    
\end{align}
In particular, at the point $p=\overline{p}_1$
$$v_{n-k+1}= (1-\overline{p}_1)^{k-1} B(n-k+1,1,1-\overline{p}_1)= (1-\overline{p}_1)^{k-1}(1-(n-k+1)(1-\overline{p}_1))=0.$$
The optimal distribution changes continuously with $p$, and therefore,
$$v_{i}>0, \qquad i\in \{n-k,n-k+2,\ldots, n-1,n\}, \,\, p\in [\overline{p}_2,\overline{p}_1],$$
for some $\overline{p}_2<\overline{p}_1$. Denote
{ \everymath={\displaystyle}
\begin{equation}\label{proof sec, equation for notation of vi' and wi'}
\begin{array}{ccll}
v_i'&=& \P(X_1=X_2=\cdots=X_i=0,X_{i+1}=\cdots=X_n=1),&\qquad i=0,1,2,\ldots,n, \\
\vspace{0.25 cm}
w_i'&=& \binom{n}{i}v_i',&\qquad i=0,1,2,\ldots,n, \\
\vspace{0.25 cm}
S_0'&=& 1,\\
\vspace{0.25 cm}
S_i'&=& \binom{n}{i}(1-p)^i,&\qquad i=1,2,\ldots,k,\\
\vspace{0.25 cm}
I'&=& \{0,1,2,\ldots, k-2,k,k+1\}.\\
\end{array} 
\end{equation}
(Note that $v_{i}'= v_{n-i}$, we are using this notation to make the calculation simple and easily relate the calculation below to the proof of Theorem \ref{resuls sec, theorem on second interval 1})
It follows from the discussion in Subsection \ref{subsection on LP} and \eqref{pre sec, equation of elements of matrix B} that
\begin{equation}\label{proof sec, equation of vi second last interval}
 v_i'=\frac{1}{\binom{n}{i}}\cdot \sum_{l=0}^k b_{i,l} \binom{n}{l}(1-p)^l, \qquad    i\in I',
\end{equation}
where 
\begin{equation}\label{proof sec, equation of b(st) second last interval}
b_{s,t} =  \frac{(-1)^t \sum_{\alpha=0}^t (-1)^\alpha \binom{t}{\alpha}\prod_{u\in I'-\{s\}}(u- \alpha)}{\prod_{u\in I'-\{s\}}(u- s)},\qquad s\in I',\, 0\leq t\leq k.   
\end{equation}
By \eqref{proof sec, equation of b(st) second last interval},
\begin{equation}\label{proof sec, equation of b(0,j) second interval last for i less then k-2}
\begin{split}
b_{0,j}=\begin{dcases}
 (-1)^j,& \qquad 0\leq j\leq k-2, \\
 (-1)^{k-1} - (-1)^{k-1}D_0,& \qquad j=k-1,\\
(-1)^{k} - (-1)^{k}\binom{k}{k-1}D_0,&\qquad j=k,
\end{dcases}
\end{split}  
\end{equation}
where
\begin{equation}\label{proof sec, equation of C0 second last}
\begin{split}
D_0&= \frac{1}{\binom{k+1}{2}}.
\end{split}    
\end{equation}
In view of \eqref{proof sec, equation of vi second last interval}, \eqref{proof sec, equation of b(0,j) second interval last for i less then k-2}, and \eqref{proof sec, equation of C0 second last}, we obtain
\begin{equation}\label{proof sec, equation of v(0)' in second last interval}
\begin{split}
v_0'&= \sum_{j=0}^{k}b_{0,j} \binom{n}{j} (1-p)^j\\
&= \sum_{j=0}^{k}(-1)^j\binom{n}{j}(1-p)^j \\
&~~~- (-1)^{k-1}\left(\binom{n}{k-1}(1-p)^{k-1} - \binom{k}{k-1}\binom{n}{k} (1-p)^{k}\right) D_0\\
&=B(n,k,1-p)- (-1)^{k-1}\binom{n}{k-1} \cdot (1-p)^{k-1} B(n-k+1,1,1-p) D_0\\
&=B(n,k,1-p)- (-1)^{k-1}\frac{\binom{n}{k-1}}{\binom{k+1}{2}} \cdot (1-p)^{k-1} B(n-k+1,1,1-p)\\
&=B(n,k,1-p)+\frac{\binom{n}{k+1}}{\binom{n-k+1}{2}}\cdot (1-p)^{k-1}B(n-k+1,1,1-p).\\
\end{split}
\end{equation}
By \eqref{proof sec, equation of v(0)' in second last interval},  and using the fact that $v_0'= v_n$, we have
\begin{equation}\label{proof sec, equation of vn in the second last interval}
v_n= B(n,k,1-p)+\frac{\binom{n}{k+1}}{\binom{n-k+1}{2}}\cdot (1-p)^{k-1}B(n-k+1,1,1-p).    
\end{equation}
By \eqref{proof sec, equation of b(st) second last interval}, for $1\leq i\leq k-2$,
\begin{equation}\label{proof sec, equation of b(i,j) second last interval for i less then k-2}
\begin{split}
b_{i,j}=\begin{dcases}
(-1)^j(-1)^i\binom{j}{i},& \qquad i\leq j\leq k-2, \\
(-1)^{k-1}(-1)^i\binom{k-1}{i} - (-1)^{k-i-1}D_{i}, & \qquad j=k-1,\\
(-1)^{k}(-1)^i\binom{k}{i} + (-1)^{k-i-1}\binom{k}{k-1}D_{i},& \qquad j=k,
\end{dcases}
\end{split}   
\end{equation}
where 
\begin{equation}\label{proof sec, equation of D1 in second last interval}
\begin{split}
D_{i}&= \frac{\binom{n-i}{k+1-i}\cdot \binom{n}{i}}{\binom{n}{k-1}\cdot \binom{n-k+1}{2}}.
\end{split}  \qquad 1\leq i\leq k-2.  
\end{equation}
It follows from \eqref{proof sec, equation of vi second last interval}, \eqref{proof sec, equation of b(i,j) second last interval for i less then k-2}, and \eqref{proof sec, equation of D1 in second last interval} that, for $1\leq i\leq k-2$,
\begin{equation*}
\begin{split}
\binom{n}{i}v_i'&= \sum_{j=0}^{k}b_{i,j} \binom{n}{j} (1-p)^j\\
&= \sum_{j=i}^{k}(-1)^j(-1)^i \binom{j}{i}\binom{n}{j}(1-p)^j\\
&~~~~ - (-1)^{k-i-1} (1-p)^{k-1} \left(\binom{n}{k-1} - \binom{k}{k-1}\binom{n}{k} (1-p)\right) D_{i}\\
&= \binom{n}{i} (1-p)^iB(n-i,k-i,1-p)\\
&~~~~~~ -(-1)^{k-i-1} (1-p)^{k-1} \binom{n}{k-1} B(n-k+1,1,1-p) D_{i},\\
\end{split}
\end{equation*}
and therefore
\begin{equation}\label{proof sec, equation of v(i)' in second last interval i <=k-2}
v_{n-i}= v_i'=(1-p)^iB(n-i,k-i,1-p) -(-1)^{k-i-1}\frac{\binom{n-i}{k+1-i}}{\binom{n-k+1}{2}}\cdot (1-p)^{k-1} B(n-k+1,1,1-p).    
\end{equation}
By \eqref{proof sec, equation of b(st) second last interval},
\begin{equation}\label{proof sec, equation of b(k,i) second last interval}
\begin{split}
b_{k,j}= \begin{dcases}
D_{k}, & \qquad j=k-1,\\
- \binom{k}{k-1 }D_k +1, & \qquad j=k,\\    
\end{dcases}
\end{split}
\end{equation}
where 
\begin{equation}\label{proof sec, equation of Dk second last interval}
\begin{split}
D_k&= \frac{\binom{k-1}{1}}{\binom{k}{2}}.
\end{split}    
\end{equation}
According to \eqref{proof sec, equation of vi second last interval}, \eqref{proof sec, equation of b(k,i) second last interval}, and \eqref{proof sec, equation of Dk second last interval},
\begin{equation*}
\begin{split}
\binom{n}{k}v_k'&= \binom{n}{k}(1-p)^k + (1-p)^{k-1}\left(\binom{n}{k-1} - \binom{k}{k-1}\binom{n}{k} (1-p)\right) D_k\\
&= \binom{n}{k}(1-p)^k +\binom{n}{k-1} (1-p)^{k-1}B(n-k+1,1,1-p) D_k\\
&=\binom{n}{k}(1-p)^k +\frac{\binom{n}{k-1}\cdot \binom{k-1}{1}}{\binom{k}{2}}\cdot (1-p)^{k-1}B(n-k+1,1,1-p),\\
\end{split}
\end{equation*}
and therefore
\begin{equation}\label{proof sec, equation of v(k)' in second last interval}
v_{n-k}= v_k' =(1-p)^k + \frac{\binom{n-k}{1}}{\binom{n-k+1}{2}}\cdot (1-p)^{k-1}B(n-k+1,1,1-p).   
\end{equation}
By \eqref{proof sec, equation of b(st) second last interval},
\begin{equation}\label{proof sec, equation of b(k+1,i) second last interval}
\begin{split}
b_{k+1,j}= \begin{dcases}
-D_{k+1}, & \qquad j=k-1,\\
\binom{k}{k-1 }D_{k+1}, & \qquad j=k,\\    
\end{dcases}
\end{split}
\end{equation}
where 
\begin{equation}\label{proof sec, equation of D(k+1) second last interval}
\begin{split}
D_{k+1}&= \frac{1}{\binom{k+1}{2}}.
\end{split}    
\end{equation}
Finally \eqref{proof sec, equation of vi second last interval}, \eqref{proof sec, equation of b(k+1,i) second last interval}, and \eqref{proof sec, equation of D(k+1) second last interval} imply
\begin{equation*}
\begin{split}
\binom{n}{k+1}v_{k+1}'&= - (1-p)^{k-1}\left(\binom{n}{k-1} - \binom{k}{k-1}\binom{n}{k} (1-p)\right) D_{k+1}\\
&= -\binom{n}{k-1}(1-p)^k B(n-k+1,1,1-p) D_{k+1}\\
&= -\frac{\binom{n}{k-1}}{\binom{k+1}{2}}(1-p)^k B(n-k+1,1,1-p),
\end{split}
\end{equation*}
and thus
\begin{equation}\label{proof sec, equation of v(k+1)' in second last interval}
v_{n-k-1}= v_{k+1}' =-\frac{1}{\binom{n-k+1}{2}}\cdot (1-p)^{k-1} B(n-k+1,1,1-p).   
\end{equation}
By Proposition \ref{result, proposition for B(n,l,p)} and \eqref{proof sec, equation of v(k+1)' in second last interval},
\begin{equation}\label{proof sec, equation2 of v(k+1)' in second last interval}
v_{n-k-1}\geq 0, \qquad p\leq\overline{p}_1.   
\end{equation}
Using Theorem \ref{pre sec, theorem on B(I) tight bound}, \eqref{proof sec, equation of vn in the second last interval}, and \eqref{proof sec, equation2 of v(k+1)' in second last interval}, we arrive at the desired result
$$M(n,k.p)= B(n,k,1-p)+ \frac{\binom{n}{k+1}}{\binom{n-k+1}{2}}\cdot (1-p)^{k-1}B(n-k+1,1,1-p),\qquad p\in [\overline{p}_2,\overline{p}_1].$$
\vspace{0.25 cm}

\noindent 2. For $k=2$, the claim follows readily from Proposition \ref{prop for M(n,2,p) except finitely many p}. Thus, let $k\geq 4$.
Using \eqref{proof sec, equation of v(k)' in second last interval}, we have 
\begin{equation}\label{proof sec, equation of v(n-k) case 2 second last interval}
\begin{split}
 v_{n-k}&= (1-p)^k + \frac{\binom{n-k}{1}}{\binom{n-k+1}{2}}\cdot (1-p)^{k-1} B(n-k+1,1, 1-p)\\
 &= (1-p)^{k-1}\left(\frac{2}{n-k+1}-(1-p)\right)\\
 &= (1-p)^{k-1}g_{n-k}(p),
\end{split}    
\end{equation}
where
$$g_{n-k}(p)= \frac{2}{n-k+1}-(1-p).$$
By \eqref{proof sec, equation of v(i)' in second last interval i <=k-2},
\begin{equation}\label{proof sec, equation of v(n-k+2) case 2 second interval}
\begin{split}
v_{n-k+2}&= (1-p)^{k-2}B(n-k+2,2,1-p) + \frac{\binom{n-k+2}{3}}{\binom{n-k+1}{2}}\cdot (1-p)^{k-1} B(n-k+1,1,1-p)\\
&= (1-p)^{k-2}\left(B(n-k+2,2,1-p)+ \frac{n-k+2}{3}\cdot (1-p) B(n-k+1,1,1-p)\right)\\
&=(1-p)^{k-2}\left(1- \frac{2}{3}\binom{n-k+2}{1}(1-p) + \frac{1}{3}\binom{n-k+2}{2}(1-p)^2\right)\\
&= (1-p)^{k-2}g_{n-k+2}(p),
\end{split}  
\end{equation}
where 
\begin{equation}\label{proof sec equation of g(n-k+2)}
g_{n-k+2}(p)=1- \frac{2}{3}\binom{n-k+2}{1}(1-p) + \frac{1}{3}\binom{n-k+2}{2}(1-p)^2.    
\end{equation}
Employing \eqref{proof sec, equation of v(i)' in second last interval i <=k-2}, we obtain
\begin{equation}\label{proof sec, equation of v(n-k+3) case 2 second interval}
\begin{split}
v_{n-k+3}&= (1-p)^{k-3}B(n-k+3,3,1-p) - \frac{\binom{n-k+3}{4}}{\binom{n-k+1}{2}}\cdot (1-p)^{k-1} B(n-k+1,1,1-p)\\
&= (1-p)^{k-3}\left(B(n-k+3,3,1-p)- \frac{\binom{n-k+3}{4}}{\binom{n-k+1}{2}}\cdot(1-p)^2 B(n-k+1,1,1-p)\right)\\
&=(1-p)^{k-3}\cdot g_{n-k+3}(p),
\end{split}  
\end{equation}
where 
\begin{equation}\label{proof sec, equation of g(n-k+3)}
 g_{n-k+3}(p) =1-\binom{n-k+3}{1}(1-p) + \frac{5}{6} \binom{n-k+3}{2}(1-p)^2- \frac{1}{2}\binom{n-k+3}{3}(1-p)^3. 
\end{equation}
By \eqref{proof sec, equation2 of v(k+1)' in second last interval}-\eqref{proof sec, equation of g(n-k+3)}, considered as function of $p$ on the interval $\left[1-\frac{2}{n-k+1},1-\frac{1}{n-k+1}\right]$:
\begin{itemize}
\item $v_{n-k-1}\geq 0$ throughout the interval;
\item $v_{n-k}$ vanishes at the left endpoint of the interval and is increasing throughout;
\item $v_{n-k+2}$ attains it minimum, which is positive, at the left endpoint;
\item $v_{n-k+3}$ is strictly positive at the right endpoint and strictly negative at the left endpoint. It vanishes only at the point $\overline{p}_2$ of this interval;
\item $v_{n-k+i}$ is non-negative throughout the sub-interval $[\overline{p}_2,\overline{p}_1]$ for $i=4,5,\ldots,k$ by Theorem \ref{pre sec, theorem on basis change}.
\end{itemize}
Using Theorem \ref{pre sec, theorem on B(I) tight bound}, we conclude the proof.

\section{Numerical Data} \label{numerical data}

\tcb{\rcb{In this final section, we conduct numerical studies \tcb{and present our key observations from these studies as eight open conjectures. }}}

\indent In the sequel, the data we will consider is for \tcb{the following values of $k$ and $n$}:
\begin{enumerate}
\item $k=4$ and $6\leq n\leq 500$.
\item $k=6$ and $8\leq n\leq 200$.
\item $k=8$ and $10\leq n\leq 100$.
\item $k=10$ and $12\leq n\leq 50$.
\end{enumerate} 
All eight conjectures hold for this observed data.

Let $2\leq k\leq n-1$ with even $k$. Recall that a basis $B$ is dual feasible if it consists of pairs of consecutive columns of $A$, say $a_1, a_1+1, \ldots, a_{k/2}, a_{k/2}+1$, and in addition column~$n$. By abuse of language, we will refer to the set of column numbers
$$I= \{a_1,a_1+1,a_2,a_2+1,\ldots, a_{k/2},a_{k/2}+1,n\}\subset \{0,1,2,\ldots,n\},$$
as the basis. In all collected data, for every dual feasible basis $I$, the set 
$$L(I)=\{p\in [0,1]\, :\, v_i\geq 0,\, i\in I\}$$
is a closed interval. We will thus refer to $L(I)$ as the {\it interval corresponding to} $I$. For the vast majority of bases $I$, the interval $L(I)$ is empty or degenerate (see Proposition \ref{Observation on realizable bases R} below). We will be interested in {\it realizable bases} $-$ those for which~$L(I)$ is of positive length. 

\subsection{Ordering the Realizable Bases}
Recall that the lexicographic order $\preceq$ on $\Z^m$ is defined as follows: for any two points $\textbf{x}=(x_1,\ldots, x_m)$ and $\textbf{y}=(y_1,\ldots,y_m)$ in $\Z^m$, we have $\textbf{x}\preceq \textbf{y}$ if either
\begin{enumerate}
\item for some $1\leq i\leq n$ we have $x_l=y_l$ for all $1\leq l\leq i-1$ and $x_i<y_i$, or
\item $\textbf{x}=\textbf{y}$.
\end{enumerate}
\tcb{Denote by $\mathcal{R}$ the set of all realizable bases.} The lexicographic order on $\Z^m$ induces an order relation on the set $\mathcal{R}$, also denoted by~$\preceq$. We simply order the elements of each basis in increasing order and view the basis as a tuple in $\Z^{k+1}$.

Let $\mathcal{L}$ be the set of all closed sub-intervals of positive length of $[0,1]$. We define an additional order relation, denoted by $\preceqdot$, on $\mathcal{L}$, as follows: given two intervals $J_1=~[a_1,b_1]$,~\,$J_2=[a_2,b_2]\in \mathcal{L}$, we have $J_1\precdott J_2$ if $J_1$ is to the left of $J_2$, namely $b_1\leq a_2$, or $J_1= J_2$. Note that, unlike $\preceq$, the order $\preceqdot$ is not total. 

\tcb{We now fix $n$ and $k$} and write $\mathcal{R}=\{I_1, I_2,\ldots, I_N\}$, ordered lexicographically:
$$I_1 \prec I_2 \prec \cdots \prec I_N.$$
\begin{example}\emph{
\begin{enumerate}
\item For $k=2$ and $n=10$, we have
\begin{equation*}
 \begin{split}
 \mathcal{R}=& \big\{\{0,1,10\}, \{1,2,10\},\{2,3,10\},\{3,4,10\},\{4,5,10\},\{5,6,10\},\{6,7,10\},\\
 &~~\{7,8,10\},\{8,9,10\}\big\}.   
 \end{split}   
\end{equation*} 
\item For $k=4$ and $n=6$, we have 
\begin{equation*}
 \begin{split}
 \mathcal{R}=& \big\{\{0,1,2,3,6\},\{0,1,3,4,6\},\{1,2,3,4,6\},\{1,2,4,5,6\},\{2,3,4,5,6\}\big\}.  \end{split}   
\end{equation*}
\end{enumerate}
}
\end{example}

\begin{conjecture}\label{conjecture on the order of I}\emph{
The mapping $L:\mathcal{R} \to \mathcal{L}$, taking each realizable basis to the corresponding interval, is order-preserving, namely
$$I, I' \in \mathcal{R}, \, I\prec I' \implies L(I) \precdott L(I').$$
\,\,\,\,\,In other words, writing $I_i=[c_i,d_i]$ for $1\leq i\leq N$, we have
$$0=c_0< d_0=c_1< d_1=c_2< \cdots =c_N< d_N=1.$$
}  
\end{conjecture}

\begin{example}\emph{
For $k=4$ and $n=10$, there are $28$ dual feasible bases, $12$ of which are realizable. Table \ref{table k=4 and n=10} lists all $I_i$-s, as well as the intervals $L(I_i)$ and their lengths $|L(I_i)|$. 
(Here, all numerical values are rounded to two or three significant digits in their decimal expansion.) The graph of the function $M(10,4,p)$ is plotted in Figure \ref{fig11}. The intervals~$L(I_i)$ are depicted alternately in red and cyan to emphasize the piecewise polynomiality of the function.}
\end{example}

\begin{table} 
    \centering
\begin{tabular}{|   C{1cm}| C{3cm} | C{3cm} | C{3cm} |}
 \hline
$i$ & $I_i$&  $L(I_i)$& $|L(I_i)|$ \\
 \hline
1& \{0,1,2,3,10\}&$~~~~~[0,1/7]$&1/7\\
 \hline
 2& \{0,1,3,4,10\}&~$[1/7,0.23]$& 0.090\\
 \hline
 3&  \{1,2,3,4,10\}&$[0.23, 0.26]$&0.022\\
 \hline
 4&  \{1,2,4,5,10\}&$[0.26,0.37]$&0.117\\
 \hline
 5& \{2,3,4,5,10\}&$[0,37,0.38]$&0.004\\
 \hline
6 & \{2,3,5,6,10\}&~$[0.38,1/2]$&0.123\\
 \hline
7 &\{3,4,6,7,10\} &~$[1/2,0.62]$&0.123\\
 \hline
8&\{4,5,6,7,10\}&$[0.62, 0.63]$&0.004\\
 \hline
9 &  \{4,5,7,8,10\}&$[0.63,0.74]$&0.117\\
 \hline
10 & \{5,6,7,8,10\}&$[0.74,0.77]$&0.022\\
 \hline
 11& \{5,6,8,9,10\}&~$[0.77,6/7]$&0.090\\
 \hline
12 &\{6,7,8,9,10\}&~~~~~$[6/7,1]$&1/7\\
 \hline
\end{tabular}
\vspace{0.25 cm}
\caption{Realizable bases, corresponding intervals, and their lengths for $k=4$ and $n=10$.}
\label{table k=4 and n=10}
\end{table}

\begin{figure}[H]
\centering
\includegraphics[width=0.80\linewidth]{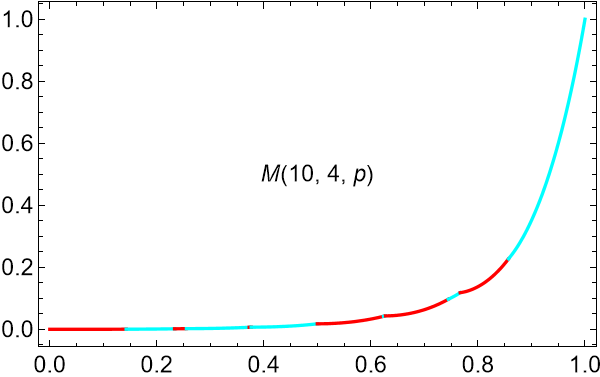}\\
\caption{A graphical depiction of $M(10,4,p)$.}
\label{fig11}
\end{figure}
\vspace{0.25 cm}

\subsection{Realizable Bases}\label{Subsection on realizable bases}
Let $I_i$ be an arbitrary fixed realizable basis. As $p$ varies from the left endpoint of $L(I_{i})=[c_i,d_i]$ to its right endpoint, the functions $v_j= v_j(p)$, $j\in I_i$ are non-negative. As we move beyond the right endpoint $d_i$, one of these functions becomes negative $-$ the solution of equation \eqref{pre sec, equation of LP primal in k-wise independent} involving only $v_j$-s with $j\in I_i$ stops being feasible, and the control passes to the next basis $I_{i+1}$, with corresponding interval $L(I_{i+1})= [c_{i+1},d_{i+1}]$, where $c_{i+1}=d_{i}$. 
\begin{conjecture}\label{conjecture on consecutive I-s}\emph{ \tcb{For at most $d_i$ breakpoints}, the transition from $I_{i}$ to $I_{i+1}$ involves a {\it minimal change} $-$ $I_{i+1}$ is obtained from $I_{i}$ by disposing of some $a_j\in I_{i}$ and adjoining $a_{j}+2$ in its stead. Namely, $I_{i}$ and $I_{i+1}$ are of the form
\begin{equation}\label{data section, equation of I(i) and I(i+1)}
\begin{array}{lll}
I_i&= &\{a_1,a_{1}+1, \ldots,a_{j-1}, a_{j-1}+1, \underline{a_{j}, a_{j}+1},\ldots, n\},\\
 I_{i+1}&=& \{a_1,a_{1}+1, \ldots,a_{j-1}, a_{j-1}+1, \underline{a_{j}+1, a_{j}+2},\ldots, n\},
\end{array}
\end{equation}
for some $1\leq j\leq k/2$.
}
\end{conjecture}

We note that a minimal change may look somewhat different from \eqref{data section, equation of I(i) and I(i+1)}. For example,~$I_i$ could have contained some block of the form $l,l+1,l+2,l+3$ (in addition to other elements), and $l$ would be replaced by $l+4$. However, this is never the case in the collected data.

Sometimes, though, the basis undergoes more than one change at a point; several elements $i_1,i_2,\ldots, i_r$ leave the basis simultaneously, and $r$ other elements join the basis for some $r\geq 2$. Points $d_i$, where such a transition occurs, are {\it exceptional} \tcb{(see Section \ref{Subsection on Exceptional points}).}

\begin{conjecture}\label{conjecture on several changes 1}\emph{In the transition from $I_{i}$ to $I_{i+1}$, if several elements $i_1,i_2,\ldots,i_r$ leave the basis simultaneously, they are
replaced by $i_1+2,i_2+2,\ldots, i_r +2$.
}
\end{conjecture}

\begin{example}\emph{
\begin{enumerate}
\item For $k=4$ and $n=10$, the bases $\{2,3,5,6,10\}$ and $\{3,4,6,7,10\}$ are consecutive realizable bases. We have two simultaneous changes: $2, 5$ leave, and $4, 7$ join at the exceptional point $1/2$. 
\item For $k=6$ and $n=23$, the bases $\{6,7,10,11,14,15,23\}$ and  $\{7,8,11,12,15,16,23\}$ are consecutive realizable bases. We have three simultaneous changes: $6, 10, 14$ leave, and $8, 12, 16$ join at the exceptional point $1/2$. 
\end{enumerate}
}
\end{example}

\begin{remark}\emph{
In view of Conjecture \ref{conjecture on several changes 1}, at exceptional points the basis undergoes several simultaneous changes as in \eqref{data section, equation of I(i) and I(i+1)}.  Conjecture \ref{conjecture on the order of I} does not exclude the possibility that a basis such as $\{10, 11, 20, 21, n\}$, for example, changes to $\{11, 12, 19, 20, n\}$ or to $\{11, 12, 22, 23\}$. \tcb{However, we did not encounter any such transitions}.
}
\end{remark}

\tcb{If both Conjectures \ref{conjecture on consecutive I-s} and \ref{conjecture on several changes 1} hold, then Proposition \ref{Observation on realizable bases R} below holds.}

\newpage
\begin{prop}\label{Observation on realizable bases R}
\tcb{Assume that both Conjectures} \ref{conjecture on consecutive I-s} and \ref{conjecture on several changes 1} hold. Then
\tcb{the following two statements hold}:
\emph{
\begin{enumerate}
\item For every $n$ and even $k\leq n-2$,
\begin{equation}\label{conjecture section, inequality for R}
 |\mathcal{R}|\leq \frac{k(n-k)}{2}+1.   
\end{equation}
\item If the transitions from $I_{i}$ to $I_{i+1}$, for all $1\leq i\leq N-1$, involve minimal changes (in other words, there are no exceptional points), then 
$$|\mathcal{R}|=\frac{k(n-k)}{2}+1.$$
\end{enumerate}
}
\end{prop}
\begin{proof}
Denote  
$$S(I)= \sum_{j=1}^{k/2}a_{j},\qquad I=\{a_1, a_1+1, \ldots, a_{k/2}, a_{k/2}+1, n\}.$$
From Conjectures \ref{conjecture on consecutive I-s} and \ref{conjecture on several changes 1}, it follows that for every two consecutive realizable bases $I_i, I_{i+1}$, the difference $S(I_{i+1})-S(I_{i})$ is twice the number of elements of $I_i$ replaced by others in $I_{i+1}$. In particular,
$$S(I_{i+1})- S(I_i)\geq 2, \qquad 1\leq i\leq N-1.$$
Hence
\begin{equation}\label{data section equation of S(IN)- S(I1)}
S(I_N)- S(I_1)= \sum_{i=1}^{N-1}\left(S(I_{i+1})- S(I_{i})\right)\geq 2(N-1).    
\end{equation}
By Thereoms \ref{resuls sec, theorem on first half} and \ref{resuls sec, theorem on last half},
$$I_1= \{0,1,\ldots, k-1,n\}, \qquad I_N= \{n-k, n-k+1,\ldots,n-1,n\}.$$ 
We then have
\begin{equation}\label{data section equation of S(I1)}
\begin{split}
 S(I_1)= 0 + 2 + \cdots + (k-2)= \frac{k}{2}(k-2),
\end{split}   
\end{equation}
and 
\begin{equation}\label{data section equation of S(IN)}
\begin{split}
 S(I_N)&= (n-k) + (n- k + 2) + \cdots +(n-2)\\
 &= nk - (2+ 4+\cdots + k)= nk -k \left(\frac{k}{2}+1\right).
\end{split}   
\end{equation}
By \eqref{data section equation of S(IN)- S(I1)}, \eqref{data section equation of S(I1)}, and \eqref{data section equation of S(IN)}, we obtain
$$N\leq 1 + \frac{1}{2}(S(I_N)- S(I_1))= 1 + \frac{k(n-k)}{2},$$
which is equivalent to \eqref{conjecture section, inequality for R}.
Now, in the case when there is no exceptional point, we have $S(I_{i+1})- S(I_i)=2$ for all $1\leq i\leq N-1$. \tcb{We thus} have equality in \eqref{data section equation of S(IN)- S(I1)}, and the \tcb{second claimed statement follows}.
\end{proof}
\vspace{0.25 cm}

\subsection{Exceptional Points}\label{Subsection on Exceptional points}
A pair $(n,k)$ is an {\it exceptional pair} if there exists at least one exceptional point for $(n,k)$.
Numerical data suggest that, for a fixed $k$, exceptional pairs are quite rare (see Table \ref{number of bad integers}).
\begin{table}[H] 
    \centering
\begin{tabular}{|   C{2cm}| C{4cm} | C{5cm} |}
 \hline
$k$ & Range of $n$& Number of exceptional pairs $(n,k)$\\
 \hline
\multirow{5}{*}{4}& ~~~[6,100] & 19 \\
& [101,200] & 15 \\
& [201,300] & 13 \\
& [301,400] & 13 \\
& [401,500] & 11 \\
\hline
\multirow{2}{*}{6}& ~~~[8,100] & ~6 \\
& [101,200] & ~3 \\
\hline
8&~[10,100]& ~0\\
\hline
10&~~[12,50]&~0\\
\hline
\end{tabular}
\vspace{0.25 cm}
\caption{Statistics for exceptional pairs.}
\label{number of bad integers}
\end{table}

\begin{conjecture}\label{obs9}\emph{
All the exceptional points are rational.
}  
\end{conjecture}
Table \ref{rational degenerate points for k=4} lists all exceptional points for $n$ in the range $[7,100]$ and $k=4$. \tcb{For $k=6$ in the range $[8,200]$}
$$23, 28, 35, 67, 76, 87, 135, 148, 163,$$
are the only $n$-s for which $(n,6)$ is exceptional. For each of these, $p=1/2$ is the only exceptional point.
\begin{table}[!htbp] 
    \centering
\begin{tabular}{|C{2.1cm} | L{5cm} | C{2.1cm} | L{5cm} |}
 \hline
$n$& Exceptional points &$n$& Exceptional point\\
 \hline
10&1/2& 52&2/5,\,\, 3/5\\
 \hline
 11&1/3, \,\, 2/3& 56&1/3, \,\, 2/3\\
 \hline
 17&1/2& 65&1/6, \,\, 1/2, \,\, 5/6\\
 \hline
 27&2/5, \,\, 3/5& 66&1/4, \,\, 3/4\\
 \hline
 29&1/6, \,\, 1/3, \,\, 2/3, \,\, 5/6& 77& 1/5, \,\, 3/10, \,\, 7/10, \,\, 4/5\\
 \hline
 34&1/4, \,\, 3/4& 82&1/2\\
 \hline
 37&1/2& 83&4/9,  \,\, 5/9\\
 \hline
 50&1/2& 92&1/3,  \,\, 2/3\\
 \hline
  51&1/7, \,\, 3/7, \,\, 4/7, \,\, 6/7& 100&1/7, \,\, 2/7, \,\, 5/7, \,\, 6/7\\
 \hline
\end{tabular}
\vspace{0.25 cm}
\caption{Exceptional points for $6\leq n\leq 100$ for $k=4$.}
\label{rational degenerate points for k=4}
\end{table}
\vspace{0.25 cm}

\subsection{Symmetry of Bases and of Intervals}
Theorems \ref{resuls sec, theorem on first half} and \ref{resuls sec, theorem on last half} show that $L(I_1)$ and $L(I_N)$ are mirror images of each other with respect to the point $p=1/2$. The same holds by Theorems \ref{resuls sec, theorem on second interval 1} and \ref{resuls sec, theorem on second last interval 1} for $L(I_2)$ and $L(I_{N-1})$. Moreover, the bases themselves enjoy a similar symmetry property. Namely, ignoring the number $n$ in the bases (as it belongs to all of them), $I_1$ and $I_N$ are mirror images with respect to the number $\frac{n-1}{2}$, and so are $I_2$ and $I_{N-1}$. These symmetries hold in general \tcb{for the collected data.}

\begin{conjecture}\label{conjecture on the symmetry of I}\emph{
\begin{enumerate}
\item If
$$I_i= \{a_1, a_1+1,a_2,a_2+1,\ldots, a_{k/2}, a_{k/2}+1,n\},$$
where $1\leq i\leq N$, then
$$I_{N+1-i}=\{n-a_{k/2}-2, n-a_{k/2}-1,\ldots, n-a_1-1, n-a_1,n\}.~~~~~~~~~~~$$
\item Denoting $L(I_i)=[c_i,d_i]$ for each $i$, we have $c_{N+1-i}=1-d_i$ and $d_{N+1-i}= 1- c_i$.
\end{enumerate}
}
\end{conjecture}
This implies
\begin{conjecture} \label{obs11} \emph{
If $d$ is an exceptional point for some $(n,k)$, then so is $1-d$.}
\end{conjecture}
\vspace{0.25 cm}

\subsection{The Spread of the Bases}
Let $I =\{a_1, a_1+1,a_2,a_2+1,\ldots, a_{k/2}, a_{k/2}+1,n\}$ be a dual feasible basis. Since $n$ belongs to all such bases, we ignore it and define the {\it spread} of~$I$ by
$$\mathrm{Sp}(I)= |a_{k/2}+1-a_1|.$$ 
By Theorems \ref{resuls sec, theorem on first half} and \ref{resuls sec, theorem on last half}, $I_1$ and $I_N$ are the densest realizable bases, i.e., 
\begin{equation}\label{conjecture section equation of Sp}
\mathrm{Sp}(I_1)=\mathrm{Sp}(I_N)= k-1=\min\{\mathrm{Sp}(I_i),\, :\, i=1,2,\ldots,N\}.    
\end{equation}
There may be several realizable bases other than $I_1$ and $I_N$ with spread $k-1$. For example, for $n=10$ and $k=4$, we have $I_3= \{1,2,3,4,10\}$, $I_5= \{2,3,4,5,10\}$, $I_8=\{4,5,6,7,10\}$, and $I_{10}=\{5,6,7,8,10\}$, all of which are of minimal spread (see Table \ref{table k=4 and n=10}). 
It follows from Conjecture~\ref{conjecture on the symmetry of I} that
$$\mathrm{Sp}(I_i)= \mathrm{Sp}(I_{N+1-i}),\qquad i=1,2,\ldots, N.$$

The spread of a random dual feasible basis is usually not much less than $n$. However, \tcb{our numerical studies suggest} that all realizable bases are relatively quite concentrated; namely, the spread is much smaller than $n$. In fact, for fixed $n$ and $k$, the {\it maximum spread} $\mathrm{Msp}(n,k)$ is defined by
$$\mathrm{Msp}(n,k)=\max\{\mathrm{Sp}(I_i)\, :\, i=1,2,\ldots,N\}.$$

\begin{conjecture}\label{obs12}\emph{
For each fixed $k$, $\mathrm{Msp}(n,k)$ is a monotonic non-decreasing function of $n$.}
\end{conjecture}

Tables \ref{spreadk=4}, \ref{spread k=6}, \ref{spread k=8}, \ref{spread k=10} list $\mathrm{Msp}(n,4)$, $\mathrm{Msp}(n,6)$, $\mathrm{Msp}(n,8)$, and $\mathrm{Msp}(n,10)$, respectively, for the ranges of $n$ in our \tcb{collected data}.

\begin{conjecture}\label{obs13}\emph{
\begin{equation}\label{conjecture sec, formula for spread for k=4}
\mathrm{Msp}(n,4)= \left\lfloor\sqrt{n-2}\right\rfloor+2, \qquad n\geq 6.    
\end{equation}
}
\end{conjecture}

Naturally, as $k$ grows, the function $\mathrm{Msp}(n,k)$, considered as a function of $n$, grows faster. However, the \tcb{collected data does not give such a simple formula as in} \eqref{conjecture sec, formula for spread for k=4}. Nor is it enough to guess whether, for a fixed $k\geq 6$, the function grows as $\Theta(\sqrt{n})$ or faster.

\begin{subtables}
\begin{table}[H] 
    \centering
\begin{tabular}{| C{2cm}| C{2cm} |C{2cm} |C{2cm} | C{2cm} |C{2cm} |}
 \hline
$\mathrm{Msp}(n,4)$ & Range of $n$ & $\mathrm{Msp}(n,4)$ &Range of $n$& $\mathrm{Msp}(n,4)$& Range of $n$\\
 \hline
~4 &~[6,10]&11&~[83,101]&18&[258,290]\\
\hline
~5 &[11,17]&12&[102,122]&19&[291,325]\\
\hline
~6 &[18,26]&13&[123,145]&20&[326,362]\\
\hline
~7 &[27,37]&14&[146,170]&21&[363,401]\\
\hline
~8 &[38,50]&15&[171,197]&22&[402,442]\\
\hline
~9 &[51,65]&16&[198,226]&23&[443,485]\\
\hline
10 &[66,82]&17&[227,257]&24&[486,500]\\
\hline
\end{tabular}
\vspace{0.25 cm}
\caption{The spread for $6\leq n\leq 500$ and $k=4$.}
\label{spreadk=4}
\end{table}
\begin{table}[H] 
    \centering
\begin{tabular}{| C{2cm}| C{2cm} |C{2cm} |C{2cm} | C{2cm} |C{2cm} |}
 \hline
$\mathrm{Msp}(n,6)$ & Range of $n$ & $\mathrm{Msp}(n,6)$ &Range of $n$& $\mathrm{Msp}(n,6)$& Range of $n$\\
 \hline
~6 &~~~[8,9]&13&~[42,49]&20&[110,121]\\
\hline
~7 &[10,13]&14&~[50,57]&21&[122,135]\\
\hline
~8 &[14,17]&15&~[58,67]&22&[136,148]\\
\hline
~9 &[18,23]&16&~[68,76]&23&[149,163]\\
\hline
10 &[24,28]&17&~[77,87]&24&[164,177]\\
\hline
11 &[29,35]&18&~[88,97]&25&[178,193]\\
\hline
12 &[36,41]&19&[98,109]&26&[194,200]\\
\hline
\end{tabular}
\vspace{0.25 cm}
\caption{The spread for $8\leq n\leq 200$ and $k=6$.}
\label{spread k=6}
\end{table}
\begin{table}[H]
    \centering
\begin{tabular}{| C{2cm}| C{2cm} |C{2cm} |C{2cm} | C{2cm} |C{2cm} |}
 \hline
$\mathrm{Msp}(n,8)$ & Range of $n$ & $\mathrm{Msp}(n,8)$ &Range of $n$& $\mathrm{Msp}(n,8)$& Range of $n$\\
 \hline
~8 &[10,11]&14&[29,32]&20&~[62,67]\\
\hline
~9 &[12,13]&15&[33,37]&21&~[68,75]\\
\hline
10 &[14,17]&16&[38,43]&22&~[76,82]\\
\hline
11 &[18,20]&17&[44,48]&23&~[83,90]\\
\hline
12 &[21,23]&18&[49,54]&24&~[91,98]\\
\hline
13 &[24,28]&19&[55,61]&25&[99,100]\\
\hline
\end{tabular}
\vspace{0.25 cm}
\caption{The spread for $10\leq n\leq 100$ and $k=8$.}
\label{spread k=8}
\end{table}
\begin{table}[H] 
    \centering
\begin{tabular}{| C{2cm}| C{2cm} |C{2cm} |C{2cm} | C{2cm} |C{2cm} |}
 \hline
$\mathrm{Msp}(n,10)$ & Range of $n$ & $\mathrm{Msp}(n,10)$ &Range of $n$& $\mathrm{Msp}(n,10)$& Range of $n$\\
 \hline
10 &[12,12]&14&[21,23]&18&[35,37]\\
\hline
11 &[13,14]&15&[24,26]&19&[38,42]\\
\hline
12 &[15,17]&16&[27,30]&20&[43,46]\\
\hline
13 &[18,20]&17&[31,34]&21&[47,50]\\
\hline
\end{tabular}
\vspace{0.25 cm}
\caption{The spread for $12\leq n\leq 50$ and $k=10$.}
\label{spread k=10}
\end{table}
\end{subtables}

\tcb{The numerical studies} suggest that
\begin{equation}\label{conjecture section equation of L(i)}
|L(I_1)|= |L(I_N)|>\max\{|L(I_i)|\, :\, i=2,3,\ldots, N-1\}.    
\end{equation}
\tcb{The equality in \eqref{conjecture section equation of Sp} and the inequality in \eqref{conjecture section equation of L(i)} raise the following question.} 
\begin{question}\emph{
For fixed $n$ and $k$, are the sequences  $(\mathrm{Sp}(I_i))_{i=1}^N$ and $(|L(I_i)|)_{i=1}^N$ negatively correlated?}
\end{question}
In Figure \ref{fig:exampe1 correlations}, we have plotted the correlation coefficients as a function of $n$ for~$k=4,6,8,10$. Note that the correlation coefficient is usually negative for $k=4$ and is negative throughout for $k=6,8,10$.  However, it seems to converge to $0$ as $n\to \infty$.

\begin{figure}[H]
\centering
\subfloat[$k=4$.]{\label{fig:a}\includegraphics[width=0.480\linewidth]{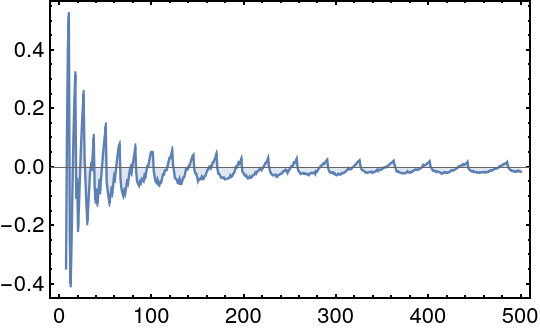}}\quad
\subfloat[$k=6$.]{\label{fig:b}\includegraphics[width=0.480\linewidth]{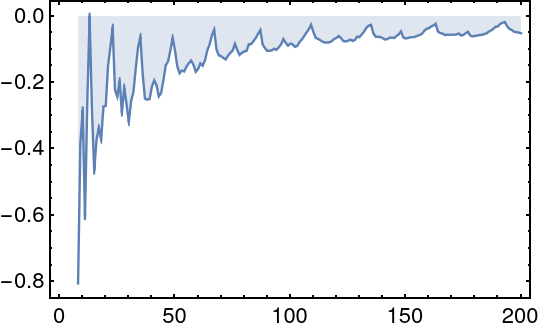}}\\
\subfloat[$k=8$.]{\label{fig:c}\includegraphics[width=0.480\textwidth]{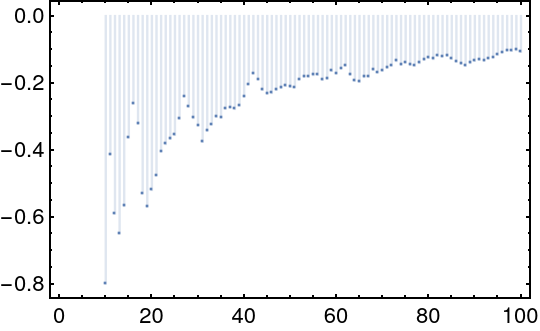}}\quad
\subfloat[$k=10$.]{\label{fig:d}\includegraphics[width=0.480\textwidth]{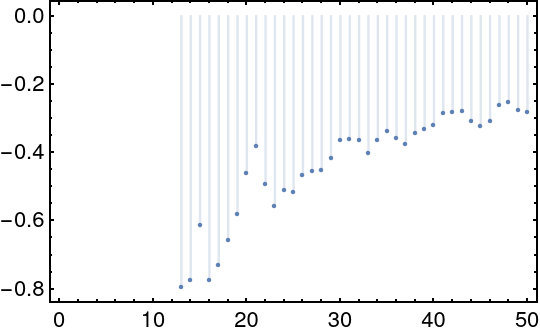}}
\caption{ Correlation between the spread of realizable bases and the lengths of the corresponding intervals.}
\label{fig:exampe1 correlations}
\end{figure}
\vspace{0.25 cm}
\tcb{As we have previously noted, we hope that the above conjectures 
will help future researchers with the very difficult task of deriving further exact formulas for $M(n,k,p)$ over the whole interval $[0,1]$.}
\bigskip

\noindent \tcb{\textbf{Acknowledgements}: We wish to thank J. Dudek, R. Pemantle, and A. Yadin for helpful discussions:}

\newpage
\bibliographystyle{plain}

\end{document}